\documentclass[final,hidelinks]{dmtcs-episciences}
\usepackage{amssymb, amsmath, amsthm, latexsym}
\usepackage{eucal}
\usepackage{ytableau}
\usepackage{undertilde}
\usepackage{cancel}
\usepackage{caption}
\usepackage[pdftex]{graphicx}

\newtheorem{thm}{Theorem}[section]
\newtheorem{cor}[thm]{Corollary}
\newtheorem{lem}[thm]{Lemma}
\newtheorem{prop}[thm]{Proposition}

\newtheorem{fact}[thm]{Fact}
\newtheorem{prob}[thm]{Problem}
\theoremstyle{definition}

\theoremstyle{remark}

\numberwithin{equation}{section}

\begin{document}

\title[Expressing flagged Schur functions with Gessel-Viennot]{Row bounds needed to justifiably express flagged Schur functions with Gessel-Viennot determinants}

\author{Robert A. Proctor\affiliationmark{1} \and Matthew J. Willis\affiliationmark{2}}

\affiliation{University of North Carolina, Chapel Hill, NC 27599 U.S.A. \\
University of Delaware-Georgetown, Georgetown, DE 19947 U.S.A.}


\received{2020-07-10}

\revised{2021-03-30}

\accepted{2021-04-01}

\publicationdetails{23}{2021}{1}{9}{6632}

\maketitle

\begin{abstract}
Let $\lambda$ be a partition with no more than $n$ parts.  Let $\beta$ be a weakly increasing $n$-tuple with entries from $\{ 1, ... , n \}$.  The flagged Schur function in the variables $x_1, ... , x_n$ that is indexed by $\lambda$ and $\beta$ has been defined to be the sum of the content weight monomials for the semistandard Young tableaux of shape $\lambda$ whose values are row-wise bounded by the entries of $\beta$.  Gessel and Viennot gave a determinant expression for the flagged Schur function indexed by $\lambda$ and $\beta$; this could be done since the pair $(\lambda, \beta)$ satisfied their ``nonpermutable'' condition for the sequence of terminals of an $n$-tuple of lattice paths that they used to model the tableaux.  We generalize flagged Schur functions by dropping the requirement that $\beta$ be weakly increasing.  Then for each $\lambda$ we give a condition on the entries of $\beta$ for the pair $(\lambda, \beta)$ to be nonpermutable that is both necessary and sufficient.  When the parts of $\lambda$ are not distinct there will be multiple row bound $n$-tuples $\beta$ that will produce the same set of tableaux.  We accordingly group the bounding $\beta$ into equivalence classes and identify the most efficient $\beta$ in each class for the determinant computation.  We recently showed that many other sets of objects that are indexed by $n$ and $\lambda$ are enumerated by the number of these efficient $n$-tuples.  We called these counts ``parabolic Catalan numbers''.  It is noted that the $GL(n)$ Demazure characters (key polynomials) indexed by 312-avoiding permutations can also be expressed with these determinants.

\end{abstract}

\vspace{1pc}\noindent\textbf{MSC Codes.}  05E05, 05A19

\section{Introduction}\label{intro}

Fix $n \geq 0$ and integers $\lambda_1 \geq \lambda_2 \geq ... \geq \lambda_n \geq 0$ throughout.  The Schur function $s_\lambda(x)$ in the variables $x_1, ... , x_n$ can be defined as the sum of the content weight monomial $x^{\theta(T)}$ over all semistandard tableaux $T$ of shape $\lambda$ with values from $\{ 1, 2, ... , n \} =: [n]$.  To generalize Schur functions, fix a ``flag'' of integers $0 \leq \beta_1 \leq \beta_2 \leq ... \leq \beta_n$.  Define $\mathcal{S}_\lambda(\beta)$ to be the set of all such tableaux whose values in the $i^{th}$ rows of their (English) shapes do not exceed the row bound $\beta_i$ for $i \in [n]$.  The flagged Schur function $s_\lambda(\beta;x)$ has been defined to be the multivariate generating function for the row bounded tableaux $T$ in $\mathcal{S}_\lambda(\beta)$ using the weight $x^{\theta(T)}$.  The set $\mathcal{S}_\lambda(\beta)$ is nonempty if and only if the row bound $n$-tuple $\beta$ has $\beta_i \geq i$ for $i \in [n]$; henceforth we assume that all row bound $n$-tuples $\beta$ satisfy this ``upper'' condition.

Gessel and Viennot considered $n$-tuples of nonintersecting lattice paths in \cite{GV}.  They could express a generating function for these $n$-tuples with a determinant via a cancellation argument, provided that the $n$-tuples of the terminals for the lattice paths satisfied their ``nonpermutable'' condition.  The Jacobi-Trudi identity expresses the Schur function $s_\lambda(x)$ as a determinant whose entries are homogeneous symmetric functions.  In his books \cite{St1} \cite{St2}, Stanley presented a Gessel-Viennot proof of the Jacobi-Trudi identity for $s_\lambda(x_1, x_2,...)$.  That proof converts the $n$-tuples of nonintersecting lattice paths to semistandard tableaux.  It more generally can produce a determinant expression for a skew flagged Schur function $s_{\lambda / \mu}(\beta;x)$, since it is noted that the flag condition on the row bounds $\beta$ is sufficient for the satisfaction of the nonpermutable condition for any given $\lambda$ and $\mu$.  The Gessel-Viennot method is still often used \cite{BRT} \cite{MPP} \cite{MS} \cite{Oka} to express the generating functions or the cardinalities for tableau sets of this  nature.

While limiting our attention to nonskew tableaux, we de-emphasize the flag condition and more generally consider all upper row bounds $\beta$ when forming the sets $\mathcal{S}_\lambda(\beta)$.  We continue to denote the corresponding generating functions by $s_\lambda(\beta;x)$.  In this more general context, can the Gessel-Viennot method still be employed to express $s_\lambda(\beta;x)$ with a determinant?  Our main result, Theorem \ref{theorem777.1}, presents $\lambda$-dependent conditions on the general upper $n$-tuples $\beta$ that are \emph{necessary as well as sufficient} for the corresponding $\lambda$-dependent $n$-tuples of lattice path terminals to be nonpermutable.  We then make some remarks on the ``row bound sums'' $s_\lambda(\beta;x)$ for general $\beta$ and on the computation of these sums with Gessel-Viennot determinants.

This paper is the third in a series of papers on key polynomials (which are the Demazure polynomials, or Demazure characters of type A) and flagged Schur functions.  Each of the themes running through these papers is of more interest to us than any one of the results is by itself.

For $i \in [n]$, the shape $\lambda$ has $\lambda_i$ boxes in its $i^{th}$ row.  The parts of $\lambda$ form a strictly decreasing sequence if and only if every one of the possible column lengths $1, 2, ... , n-1$ less than $n$ is present in the shape of $\lambda$.  It seems that the phenomena that arise when $\lambda$ is not strict may have received relatively little attention in the studies of Demazure polynomials and of flagged Schur functions.  For example, when $\lambda$ is not strict the Demazure polynomials are precisely indexed by the multipermutations for the quotient $W^J$ of the symmetric group $S_n$, and not by the permutations in $S_n$.  Many of the phenomena considered in these papers are trivial or vacuous when $\lambda$ is strict.  For example, there exist upper $n$-tuples $\beta' \neq \beta$ with $\mathcal{S}_\lambda(\beta') = \mathcal{S}_\lambda(\beta)$ exactly when $\lambda$ is not strict.  Let $R_\lambda$ be the set of column lengths less than $n$ that appear in the shape $\lambda$.  The central objects in this series of papers are $n$-tuples with entries from $[n]$ which have been equipped with dividers that are placed in locations between their entries which are indexed by the elements of $R_\lambda$.  In these papers we introduce several properties which may be possessed by such ``$R_\lambda$-tuples''.  In \cite{PW2} we defined the $R_\lambda$-parabolic Catalan number to be the number of ``$\lambda$-312-avoiding $R_\lambda$-permutations''.  The most interesting special kinds of $R_\lambda$-tuples are those which are also enumerated by the $R_\lambda$-parabolic Catalan numbers.  These include the most important of the $R_\lambda$-tuples that arise from the considerations in this paper, the ``gapless'' $R_\lambda$-tuples.  It is striking that the gapless $R_\lambda$-tuples had already arisen in \cite{PW2} and \cite{PW3} for different considerations in each of those two papers.  The gapless $R_\lambda$-tuples appear to be fundamental new combinatorial quantities.

In this area, how reliably does an equality for two generating functions predict equality for their underlying sets of tableaux?  Reiner and Shimozono \cite{RS} and then Postnikov and Stanley \cite{PS} obtained results concerning the possible equalities between Demazure polynomials $d_\lambda(\pi;x)$ and flagged Schur functions $s_\lambda(\beta;x)$.  In \cite{PW3} we built upon their work by showing that when $d_\lambda(\pi;x) = s_\lambda(\beta;x)$ then the equality $\mathcal{D}_\lambda(\pi) = \mathcal{S}_\lambda(\beta)$ for the underlying tableau sets also holds:  in other words, we ruled out any ``accidental'' equalities of the form $d_\lambda(\pi;x) = s_\lambda(\beta;x)$ with $\mathcal{D}_\lambda(\pi) \neq \mathcal{S}_\lambda(\beta)$.  Care must be taken to avoid mischaracterizing our main result, Theorem \ref{theorem777.1}.  While it states sharp necessary conditions that are needed to be able to justifiably apply the Gessel-Viennot method to express the row bound sum $s_\lambda(\beta;x)$ with the Gessel-Viennot determinant polynomial (denoted $gv_\lambda(\beta;x)$), it is conceivable that such an equality could nonetheless ``accidentally hold'' while those conditions are not being satisfied.  We define two upper $R_\lambda$-tuples $\beta$ and $\beta'$ to be equivalent when $\mathcal{S}_\lambda(\beta) = \mathcal{S}_\lambda(\beta')$.  Then $s_\lambda(\beta;x) = s_\lambda(\beta';x)$.  It is conceivable that accidental equalities of the form $s_\lambda(\beta;x) = s_\lambda(\beta';x)$ while $\mathcal{S}_\lambda(\beta) \neq \mathcal{S}_\lambda(\beta')$ could exist.

As we consider various kinds of upper $R_\lambda$-tuples $\beta$, we study the applicability and the efficiency of the Gessel-Viennot method for expressing the general row bound sum polynomials $s_\lambda(\beta;x)$ with determinants.  The following overview of our conclusions is organized with an outlined hierarchy of kinds of upper $R_\lambda$-tuples.  Some readers may wish to defer reading this technical summary until their second reading of this introduction, and some readers may prefer to read the items in the order I, II, A, B, (1), (2), (3).

\noindent I.  When $\beta$ is a ``gapless core'' $R_\lambda$-tuple, then there exists a flag $\varphi$ that is equivalent to it.  So the polynomial $s_\lambda(\beta;x)$ arises as an already-known flagged Schur function $s_\lambda(\varphi;x)$ and thus it is equal to the Gessel-Viennot determinant $gv_\lambda(\varphi;x)$.  However, the determinant expression $gv_\lambda(\beta;x)$ for $s_\lambda(\beta;x)$ will often have fewer total monomials among its entries than  $gv_\lambda(\varphi;x)$; see for example Proposition \ref{prop826.9}.

\noindent \hspace{3mm} A.  When $\beta$ is additionally a ``bounded platform'' $R_\lambda$-tuple, then the nonpermutable condition will be satisfied.  This is the sufficient direction of our main result, Theorem \ref{theorem777.1}.  Here the Gessel-Viennot method can be immediately applied to express $s_\lambda(\beta;x)$ as the determinant $gv_\lambda(\beta;x)$, as is stated in Corollary \ref{cor777.2}.

\noindent \hspace{6mm}  (1)  When $\beta$ is also a flag, then the sum $s_\lambda(\beta;x)$ is a flagged Schur function from the outset.  But a flag $\beta$ is probably not as efficient for the purposes of determinant evaluation as some equivalent gapless core bounded platform $R_\lambda$-tuple $\beta'$ would be.

\noindent \hspace{6mm}  (2)  In fact, for efficient evaluation of the Gessel-Viennot determinant, the best possible $R_\lambda$-tuples are the gapless $R_\lambda$-tuples.  See Proposition \ref{prop826.9}.

\noindent \hspace{6mm}  (3)  The two ``gapless'' and ``flag'' criteria are independent for bounded platform gapless core $R_\lambda$-tuples:  either, both, or neither may be possessed.

\noindent \hspace{3mm}  B.  When $\beta$ is not a bounded platform $R_\lambda$-tuple, then the nonpermutable condition will not be satisfied.  This is the necessary direction of Theorem \ref{theorem777.1}.  But this gapless core $\beta$ can nonetheless be ``pre-processed'' to produce an equivalent $\beta'$ to which the Gessel-Viennot method can be applied.  See Corollary \ref{cor777.2.5}.  When $\beta$ is a gapless core $R_\lambda$-tuple, Corollary \ref{fctinterval} describes all gapless core bounded platform $R_\lambda$-tuples $\beta'$ for which $s_\lambda(\beta;x) = gv_\lambda(\beta';x)$.

\noindent II.  When $\beta$ is not a gapless core $R_\lambda$-tuple, then $s_\lambda(\beta;x)$ cannot arise as a flagged Schur function.  In Section 8 the necessary direction of Theorem \ref{theorem777.1} will be used to remark that there are no $R_\lambda$-tuples $\beta'$ equivalent to such a $\beta$ for which the nonpermutable condition is satisfied.

Both in \cite{PW3} and in this paper we have found some ``nice'' properties that are possessed by the flagged Schur functions which do not hold (or which have not been obtained by us) for the row bound sums $s_\lambda(\beta;x)$ that arise from non-gapless core $R_\lambda$-tuples.  See the last paragraph of Section 8.  Problem \ref{open} asks if a row bound sum $s_\lambda(\beta;x)$ for an upper $R_\lambda$-tuple $\beta$ that is not a gapless core $R_\lambda$-tuple can ``accidentally'' be equal to a Gessel-Viennot determinant.

Corollary \ref{EqualDeterms} describes a potential application to determinant enumerations:  If two seemingly unrelated sets of combinatorial objects are each enumerated with determinants of binomial coefficients and many test evaluations of them agree, this corollary provides an alternative to row and column operations for proving that the determinants are equal in general.

We conclude by indicating where this material is situated and how we were led to the considerations that crystallized into our main result.  Demazure characters arose in 1974 when Demazure introduced certain $B$-modules while studying singularities of Schubert varieties in the $G/B$ flag manifolds.  For $G = GL(n)$, a Demazure polynomial $d_\lambda(\pi;x)$ can be expressed as the sum of the weight monomial $x^{\theta(T)}$ over a certain set $\mathcal{D}_\lambda(\pi)$ of semistandard tableaux of shape $\lambda$.  Flagged Schur functions arose in 1982 when Lascoux and Sch\"{u}tzenberger were studying Schubert polynomials for the flag manifold $GL(n)/B$.  Seeking a deeper understanding of the results in \cite{RS} and  \cite{PS} that related the polynomials $d_\lambda(\pi;x)$ and $s_\lambda(\beta;x)$ to each other led us to the studies described in this three paper series.  Set $N := \lambda_1 + \lambda_2 + ... + \lambda_n$.  Our main result in \cite{PW2}, Corollary 7.2, stated that the set $\mathcal{D}_\lambda(\pi)$ of ``Demazure tableaux'' is convex in $\mathbb{Z}^N$ if and only if the $R_\lambda$-permutation $\pi$ is $\lambda$-312-avoiding.  Both directions of that result were used in \cite{PW3} as we sharpened, extended, and deepened the results of \cite{RS} and \cite{PS}.  The proof of Corollary 10.4(i) of \cite{PW3} used the main result of \cite{PW2}.  When $\beta$ has the gapless core property, this corollary ruled out the accidental equalities $s_\lambda(\beta;x) = s_\lambda(\beta';x)$ with $\mathcal{S}_\lambda(\beta) \neq \mathcal{S}_\lambda(\beta')$. Corollary 10.4(i) is restated here as Fact \ref{setequality}, for use in the proof of Corollary \ref{fctinterval}.  When combined, Corollary 9.2 of \cite{PW2} and Theorem 13.1 of \cite{PW3} list nearly a dozen kinds of $R_\lambda$-tuples and phenomena that are counted by the $R_\lambda$-parabolic Catalan numbers.  The ruling out of the accidental equalities just mentioned together with the connection between flagged Schur functions and Demazure polynomials led to the enumeration of the distinct flagged Schur polynomials that is listed as Part (iv) of Theorem 13.1.  The last item in these lists, Part (vi) of Theorem 13.1 of \cite{PW3}, refers to the enumeration in Corollary \ref{cor826.12} below.  Being aware that interesting phenomena arise for the row bound sums $s_\lambda(\beta;x)$ when $\lambda$ is not strict and being familiar with gapless $R_\lambda$-tuples and the $R_\lambda$-ceiling map from \cite{PW2} and \cite{PW3} enabled us to recognize the necessity of the conditions in Theorem \ref{theorem777.1}, and then to see that these conditions extended the realm of sufficiency from flags alone.

Routine definitions appear in Section \ref{defns} and advanced definitions appear in Section \ref{advdefns}.  Our main results are stated in Section \ref{main}, after our set-up of the Gessel-Viennot mechanics has been presented in Section \ref{paths}.  Sections \ref{necc} and \ref{suff} contain the proofs of our main results.  Sections \ref{nice} and \ref{effic} consider equivalences for the row bounds $\beta$, two kinds of accidental equalities, and efficiency for the Gessel-Viennot determinants.  Section \ref{Dem} presents an application to representation theory and algebraic geometry;  there we improve a determinant expression for some Demazure polynomials that appeared in \cite{PS}.

\section{Elementary $\mathbf{\emph{n}}$-tuples and shapes, tableaux, polynomials}\label{defns}

Fix $n \geq 1$ throughout the paper.  Let $i, k \geq 0$.  Define $(i, k] := \{i+1,$ $i+2, ... , k\}$ and $[k] := \{1, 2, ... , k \}$ and so on.  Lower case Greek letters indicate tuples of non-negative integers; their entries are denoted with the same letter.  Other than $\zeta$, all of these tuples are $n$-tuples.  An $nn$-\textit{tuple} $\nu$ has $n$ \emph{entries} $\nu_i \in [n]$ indexed by $n$ \emph{indices} $i \in [n]$.  Let $P(n)$ denote the poset of $nn$-tuples ordered by entrywise comparison.  An $nn$-tuple $\varphi$ is a \textit{flag} if $\varphi_1 \leq \ldots \leq \varphi_n$.  An \emph{upper tuple} is an $nn$-tuple $\beta$ such that $\beta_i \geq i$ for $i \in [n]$.

Fix $R \subseteq [n-1]$.  Set $r := |R|$.  When $R \neq \emptyset$ denote the elements of $R$ by $q_1 < \ldots < q_r$.  Two numbers not in $R$ are $q_0 := 0$ and $q_{r+1} := n$.  We use the $q_h$ for $h \in [r]$ to specify the locations of $r$ ``dividers'' within $nn$-tuples:  Let $\nu$ be an $nn$-tuple.  On the graph of $\nu$ in the first quadrant draw vertical lines at $x = q_h + \epsilon$ for $h \in [r]$ and some small $\epsilon > 0$.  These $r$ lines separate the $r+1$ \emph{carrels} $(q_{h-1}, q_h]$ \emph{of $\nu$} for $h \in [r+1]$.  An \emph{$R$-tuple} is an $nn$-tuple that has been equipped with these $r$ dividers.  For $h \in [r+1]$ the $h^{th}$ carrel contains $q_h - q_{h-1}$ indices.  Fix an $R$-tuple $\nu$;  we portray it by $(\nu_1, ... , \nu_{q_1} \hspace{1mm} | \hspace{1mm} \nu_{q_1+1}, ... , \nu_{q_2}\hspace{1mm} | \hspace{1mm} ... \hspace{1mm} | \hspace{1mm} \nu_{q_r+1}, ... , \nu_n)$.  When $n = 9$ and $R = \{3, 8 \}$, the dots in Figure 5.1 display $\nu := (2,7,5\hspace{1mm} | \hspace{1mm}8,6,6,9,9\hspace{1mm} | \hspace{1mm}9)$.  Let $U_R(n)$ denote the subposet of $P(n)$ consisting of upper $R$-tuples.  An \emph{upper $R$-flag} is an upper flag regarded as an $R$-tuple.  Let $UF_R(n) \subseteq U_R(n)$ denote the set of upper $R$-flags.  An \emph{$R$-increasing tuple} is an $R$-tuple $\alpha$ such that  $\alpha_{q_{h-1}+1} < ... < \alpha_{q_h}$ for $h \in [r+1]$.  An example is given by $\delta$ in Table 6.1.  Let $UI_R(n) \subseteq U_R(n)$ denote the set of $R$-increasing upper tuples.  More kinds of $R$-tuples will be introduced in Sections \ref{advdefns} and \ref{effic}; the Table 4.1 directory lists the six essential kinds.

A \emph{partition} is an $n$-tuple $\lambda \in \mathbb{Z}^n$ such that $\lambda_1 \geq \ldots \geq \lambda_n \geq 0$.  Fix a partition $\lambda$ throughout.  The \textit{shape} of $\lambda$, also denoted $\lambda$, consists of $n$ left justified rows with $\lambda_1, \ldots, \lambda_n$ boxes.  We denote its column lengths by $\zeta_1 \geq \ldots \geq \zeta_{\lambda_1}$.  Since the columns were more important than the rows in \cite{PW2} and \cite{PW3}, the boxes of $\lambda$ are transpose-indexed by pairs $(j,i)$ such that $1 \leq j \leq \lambda_1$ and $1 \leq i \leq \zeta_j$.  Define $R_\lambda \subseteq [n-1]$ to be the set of distinct column lengths of $\lambda$ that are less than $n$.  Taking $R := R_\lambda$, note that for $h \in [r+1]$ one has $\lambda_i = \lambda_{i^\prime}$ when $i$ and $i^\prime$ are in the $h^{th}$ carrel $(q_{h-1}, q_{h} ]$ specified by $R$.  Here the number of rows in the shape of $\lambda$ that have length $\lambda_{q_h}$ is $q_h - q_{h-1}$.  For $h \in [r]$ the coordinates of the boxes in the $h^{th}$ \emph{cliff} of the shape $\lambda$ form the set $\{ (\lambda_i, i) : i \in (q_{h-1}, q_{h} ] \}$.  To reduce clutter, we replace `$R_\lambda$' by `$\lambda$' in subscripts and in prefixes when we are using $R := R_\lambda$ for the notions above.

A \textit{(semistandard) tableau of shape $\lambda$} is a filling of $\lambda$ with values from $[n]$ that strictly increase from north to south within a column and weakly increase from west to east within a row.  See Fig. 6.1.  Let $\mathcal{T}_\lambda$ denote the set of tableaux of shape $\lambda$.  Fix $T \in \mathcal{T}_\lambda$.  For $j \in [\lambda_1]$, we denote the one column ``subtableau'' formed by the $j^{th}$ column by $T_j$.  Here for $i \in [\zeta_j]$ the tableau value in the $i^{th}$ row is denoted $T_j(i)$.  To define the \emph{content $\Theta(T) := \theta$ of $T$}, for $i \in [n]$ take $\theta_i$ to be the number of values in $T$ equal to $i$.  Let $x_1,\ldots, x_n$ be indeterminants.  The \textit{monomial} $x^{\Theta(T)}$ of $T$ is $x_1^{\theta_1}\ldots x_n^{\theta_n}$.

Let $\beta$ be a $\lambda$-tuple.  We define the \emph{row bound set of tableaux} to be $\mathcal{S}_\lambda(\beta) := \{ T \in \mathcal{T}_\lambda : T_j(i) \leq \beta_i \text{ for } j \in [\lambda_1] \text{ and } i \in [\zeta_j] \}$.  As in Section 8 of \cite{PW3}, it can be seen that $\mathcal{S}_\lambda(\beta)$ is nonempty if and only if $\beta \in U_\lambda(n)$.  Fix $\beta \in U_\lambda(n)$.    In \cite{PW3} we introduced the \emph{row bound sum} $s_\lambda(\beta; x) := \sum x^{\Theta(T)}$, sum over $T \in \mathcal{S}_\lambda(\beta)$.  For upper $\lambda$-flags $\varphi \in UF_\lambda(n)$ we more specifically refer to the flagged Schur functions $s_\lambda(\varphi;x)$ as \emph{flag Schur polynomials}.  The ``gapless core'' Schur polynomials $s_\lambda(\beta; x)$ are defined in Section \ref{nice}.

\section{Lattice paths and Gessel-Viennot determinant}\label{paths}

We establish conventions and notations for constructions that are analogous to those presented in Section 2.7 of \cite{St1} and Section 7.16 of \cite{St2}.  First we introduce $n$-tuples of weighted lattice paths to model the tableaux in the row bound tableau set $\mathcal{S}_\lambda(\beta)$.  To obtain a close visual correspondence with tableaux, we first flip the $x$-$y$ plane containing the lattice paths vertically so that its first quadrant is to the lower right (southeast) of the origin on the page.  Our transposed matrix coordinates for the boxes in shapes can then be re-used to coordinatize the points in this first quadrant of $\mathbb{Z} \times \mathbb{Z}$:  Let $j \geq 0$ and $i \geq 1$.  The lattice point $(j,i)$ is $j$ units to the east of $(0,0)$ and $i$ units to the south of $(0,0)$.  For $j \geq 1$, the directed line segment from $(j-1,i)$ to $(j,i)$ is an \emph{easterly step of depth} $i$.  Let $l \geq j$ and $k \geq i$.  A \emph{(lattice) path with source $(j,i)$ and sink $(l,k)$} is a connected set incident to $(j,i)$ and $(l,k)$ that is the union of $l-j$ easterly steps and $k-i$ \emph{southerly steps}.  The notation $... \rightarrow (j,i) \downarrow (j,k) \rightarrow (l,k) \downarrow ...$ indicates that an eastbound path arrives at $(j,i)$, turns right and proceeds south to $(j,k)$, turns left and proceeds east to $(l,k)$, and then turns right and proceeds south.  An \emph{$n$-path} is an $n$-tuple $(\Lambda_1, ... , \Lambda_n) =: \Lambda$ of paths such that the component path $\Lambda_m$ has source $(n-m,m)$ for $m \in [n]$.  An $n$-path is \emph{disjoint} if no two component paths intersect.

Let $\beta \in P(n)$.  The $n$ points $(\lambda_1+n-1, \beta_1), (\lambda_2+n-2, \beta_2),... ,(\lambda_n,\beta_n)$ are \emph{terminals} and $(\lambda, \beta)$ is a \emph{terminal pair}.  This ``strictification'' of $\lambda$ ensures that the longitudes of the terminals are distinct.  Initially our $n$-paths $(\Lambda_1, ..., \Lambda_n)$ will use the terminals $(\lambda_1+n-1, \beta_1), (\lambda_2+n-2,\beta_2), ... , (\lambda_n, \beta_n)$ in this order as sinks for their respective components.  Given such an $n$-path $\Lambda$, as in the proof of Theorem 7.16.1 of \cite{St2} we attempt to create a corresponding tableau $T \in \mathcal{S}_\lambda(\beta)$.  Figure 3.1 presents an example for this correspondence wherein $n=8$ and $\lambda = (6, 4, 4, 4, 4, 1, 0, 0)$ and $\beta = (4,7,4,6,8,7,8,8)$.  For each $m \in [n]$ we record the weakly increasing depths of the successive easterly steps in the path $\Lambda_m$ from left to right in the boxes of the $m^{th}$ row of the shape $\lambda$:  Here the easterly step in $\Lambda_m$ from $(n-m+j-1,p)$ to $(n-m+j,p)$ is recorded as the value $p$ in the box $(j,m)$ for $T$.  The last value in the $m^{th}$ row cannot exceed $\beta_m$.  As is implicit in \cite{St2}, these values strictly increase down each column of $\lambda$ if and only if there are no intersections among the $\Lambda_m$ for $m \in [n]$.  (To see this, let $i \in [q_r -1]$ and $j \in [\lambda_{i+1}]$.  Set $t := T_j(i)$ and suppose $T_j(i+1) = t+1$.  Then the easterly steps for $T_j(i)$ and $T_j(i+1)$ are $(n-i+j-1,t) \rightarrow (n-i+j,t)$ and $(n-i-1+j-1,t+1) \rightarrow (n-i-1+j,t+1)$.  So the ``near miss'' for these tableau values translates to a ``near miss'' for the path edges.)  Let $\mathcal{LD}_\lambda(\beta)$ denote the set of such $n$-paths that are disjoint.  There is at least one such disjoint $n$-path if and only if $\beta$ is upper.  As is claimed in \cite{St2}, this recording process can be seen to be  bijective to the set $\mathcal{S}_\lambda(\beta)$.  Since it will be seen that the carrels and cliffs of $\lambda$ play a crucial role, we now determine $R_\lambda$ from $\lambda$ and regard $\beta$ as being a $\lambda$-tuple.  Summarizing:

\begin{fact}\label{fact315.5}We have $\mathcal{LD}_\lambda(\beta) \neq \emptyset$ if and only if $\beta \in U_\lambda(n)$.  For $\beta \in U_\lambda(n)$, the recording process is a bijection from the set of disjoint $n$-paths $\mathcal{LD}_\lambda(\beta)$ to the row bound tableau set $\mathcal{S}_\lambda(\beta)$.  \end{fact}

\begin{figure}[h!]
  \begin{center}
    \includegraphics[scale=1]{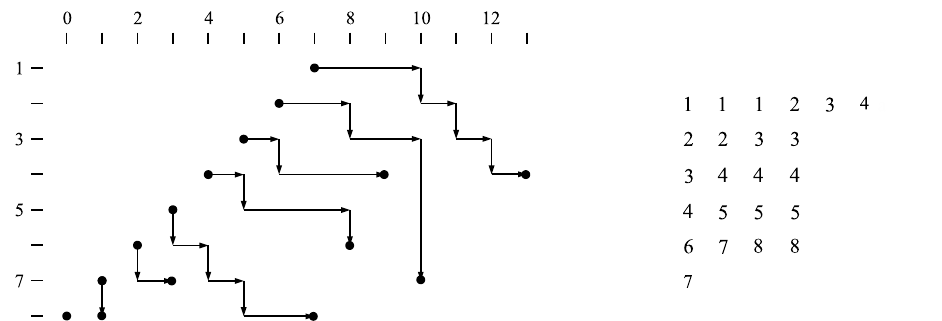}
    \caption*{Figure 3.1.  Converting a disjoint 8-path to a semistandard tableau.}
  \end{center}
\end{figure}

To visualize the sequence $(\lambda_1 + n - 1, \beta_1), (\lambda_2 + n - 2, \beta_2), ... , (\lambda_n, \beta_n)$ of $n$ terminals in the plane (as in Figure 3.1), one rotates the graph of $\beta$ (as in Figure 5.1) by $180^\circ$ around the origin and then shifts the $i^{th}$ dot by $\lambda_i+n$ to the east.

Fix $\beta \in U_\lambda(n)$.  To obtain a determinant expression for $s_\lambda(\beta;x)$ we will need to consider more general $n$-paths and introduce weights.  Let $\Lambda$ be an $n$-path with arbitrary sinks.  Assigning a weight monomial $x^{\Theta(\Lambda)}$ to $\Lambda$ in the following fashion emulates our assignment of the weight $x^{\Theta(T)}$ to a tableau $T \in \mathcal{T}_\lambda$ when $\Lambda \in \mathcal{LD}_\lambda(\beta)$, and it also extends the weight rule to all $n$-paths.  For $m \in [n]$ assign the weight $x_i$ to each easterly step of depth $i$ in the path $\Lambda_m$ and multiply these weights over its easterly steps.  Then also multiply the weights of the $n$ component paths of $\Lambda$ to produce a monomial, denoted $x^{\Theta(\Lambda)}$.  When the sinks of $\Lambda$ are the terminals from $(\lambda, \beta)$ in their usual order, the multivariate generating function $\sum_{  \Lambda \in \mathcal{LD}_\lambda(\beta)  } x^{\Theta(\Lambda)}$ is clearly our row bound sum $s_\lambda(\beta;x)$.

Let $j, l \geq 0, i, k \geq 1$, and set $u := l-j$.  The \emph{complete homogeneous symmetric function} $h_u(i,k;x)$ in the variables $x_i, x_{i+1}, ... , x_k$ is defined when $u < 0$ by $h_u(i,k;x) := 0$, and when $u \geq 0$ by $h_u(i,k;x) := \sum x_{t_1}\cdots x_{t_u}$, sum over $i \leq t_1 \leq ... \leq t_u \leq k$.  (When $i \leq k$ and $u = 0$ conventions imply $h_0(i,k;x) = 1$.)  This is the sum of the weights that are assigned to just one path as it varies over all paths from $(j,i)$ to $(l,k)$.

We next consider $n$-paths that use the same terminals, but in a permuted order, for their list of sinks.  Let $\pi$ be a permutation of $[n]$.  Let $\pi.(\lambda, \beta)$ denote the list of terminals $(\lambda_{\pi_1}+n-\pi_1, \beta_{\pi_1}), \\ (\lambda_{\pi_2}+n-\pi_2, \beta_{\pi_2}), ... , (\lambda_{\pi_n}+n-\pi_n, \beta_{\pi_n})$.  Let $\mathcal{LD}_\lambda(\beta;\pi)$ denote the set of disjoint $n$-paths $(\Lambda_1, ... , \Lambda_n)$ with respective sinks $\pi.(\lambda,\beta)$.  The terminal pair $(\lambda, \beta)$ is \emph{nonpermutable} \cite{GV} if $\mathcal{LD}_\lambda(\beta;\pi)$ is empty when $\pi$ is not the identity $(1,2,...,n)$.

Here is our nonskew version of Theorem 2.7.1 of \cite{St1}; as in Theorem 7.16.1 of \cite{St2} we have replaced the disjoint $n$-paths with the corresponding tableaux:

\begin{prop}\label{prop315.6}Let $\beta \in U_\lambda(n)$.  If the terminal pair $(\lambda, \beta)$ is nonpermutable, then the row bound sum $s_\lambda(\beta;x)$ is given by the $n \times n$ determinant $| h_{\lambda_j-j+i}(i,\beta_j;x) |$.  \end{prop}

\noindent To produce this expression with Theorem 2.7.1 of \cite{St1}, use the remark above that expressed $s_\lambda(\beta;x)$ as the $\mathcal{LD}_\lambda(\beta)$ generating function and note that $(\lambda_j+n-j)-(n-i) = \lambda_j-j+i$.  Theorem 2.7.1 was proved with a signed involution pairing cancellation argument, as in \cite{GV}.  We refer to this argument as the G-V method and to this determinant as the G-V determinant for $\beta$.

\section{Main results}\label{main}

As noted in Table 4.1, the technical definitions of the notions of ``gapless core'' and ``bounded by platform'' appearing in the results below are given in the next section.  Our main result combines our Propositions \ref{prop434.7} and \ref{prop581.6}:

\begin{thm}\label{theorem777.1}Let $\lambda$ be a partition and let $\beta$ be an upper $\lambda$-tuple.  The terminal pair $(\lambda, \beta)$ is nonpermutable if and only if $\beta$ is a gapless core $\lambda$-tuple that is bounded by its platform. \end{thm}

\noindent Since we will note that the gapless core $\lambda$-tuples that are bounded by their platforms are the upper $\lambda$-flags when $\lambda$ is strict, in that case this theorem says that the upper $\lambda$-flags are the \emph{only} upper $\lambda$-tuples which produce nonpermutable terminal pairs.  So in the strict $\lambda$ case this theorem provides the converse to Stanley's parenthetical remark in Theorem 2.7.1 of \cite{St1}.

Under the circumstances of the theorem we can employ the G-V method from Proposition \ref{prop315.6}:

\begin{cor}\label{cor777.2}Let $\lambda$ be a partition and let $\beta$ be an upper $\lambda$-tuple.  If $\beta$ is a gapless core $\lambda$-tuple that is bounded by its platform, then $s_\lambda(\beta;x) = | h_{\lambda_j-j+i}(i,\beta_j;x) |$. \end{cor}

\noindent The converse to this result is open; see Problem \ref{open}(i) below.

If $\beta$ is a gapless core $\lambda$-tuple that is not bounded by its platform, then the polynomial $s_\lambda(\beta;x)$ for such a $\beta$ can nonetheless be computed with a determinant.  The map $\Delta_\lambda$ used to convert $\beta$ to $\delta$ for the following result is also defined in the next section:

\begin{cor}\label{cor777.2.5}Let $\lambda$ be a partition and let $\beta$ be an upper $\lambda$-tuple.  Set $\delta := \Delta_\lambda(\beta)$.  If $\beta$ is a gapless core $\lambda$-tuple, then $s_\lambda(\beta;x) = | h_{\lambda_j-j+i}(i,\delta_j;x) |$. \end{cor}

\noindent The proof of this corollary is contained in the proof of the more general Corollary \ref{fctinterval}.

\begin{figure}[h!]
\begin{center}
\begin{tabular}{ccc}
$\lambda$-Tuples and their sets & \hspace{9mm} Section defined  & Terminology  \\  \hline
\color{white}asdf & & \\
$\beta \in U_\lambda(n)$ & \hspace{9mm} 2 &  upper $\lambda$-tuple \\
$\varphi \in UF_\lambda(n)$  & \hspace{9mm} 2 &  upper $\lambda$-flag \\
$\alpha \in UI_\lambda(n)$  & \hspace{9mm} 2 &  $\lambda$-increasing upper tuple \\
\vspace{-.65pc}\\
$\gamma \in UG_\lambda(n)$  & \hspace{9mm} 5 &  gapless $\lambda$-tuple \\
$\eta \in UGC_\lambda(n)$  & \hspace{9mm} 5 &  gapless core $\lambda$-tuple \\
$\xi \in UBP_\lambda(n)$  & \hspace{9mm} 5 &  $\lambda$-tuple bounded by platform
\end{tabular}
\caption*{Table 4.1.  Kinds of upper $R_\lambda$-tuples}
\end{center}
\end{figure}

\vspace{-1.5pc}

\section{Advanced $\mathbf{\emph{R}}$-tuples}\label{advdefns}

As in \cite{PW2} and \cite{PW3}, we first distill the crucial information from an upper $R$-tuple $\beta$ into a skeletal ``critical'' substructure.  Doing this was (and still is) motivated by the tableau considerations that are presented in the last paragraph of this section.  As we distill this information we define two functions from $U_R(n)$ to $U_R(n)$.  To preview, scanning an $R$-tuple $\beta$ within each of its carrels from the right, an entry of it will be a ``critical entry'' if it is either the rightmost entry in its carrel, or if it is smaller than the closest critical entry to its right by an amount that exceeds its distance from that critical entry.  Launching a running example, take $n := 9, R := \{3, 8 \}$, and $\beta := (2,7,5 \hspace{1mm}|\hspace{1mm} 8,6,6,9,9\hspace{1mm}|\hspace{1mm}9)$.  Here there are $r+1 = 3$ carrels; see Figure 5.1.  There the entries of this $\beta$ are denoted with dots.  We will define the \emph{$R$-core} and \emph{$R$-platform} maps $\Delta_R$ and $\Xi_R$ on $U_R(n)$ by constructing their images $\delta$ and $\xi$ for $\beta$.  The entries of $\delta$ and $\xi$ for our example will be denoted with dashes in the left and right parts of Figure 5.1.

\begin{figure}[h!]
  \begin{center}
    \includegraphics[scale=.95]{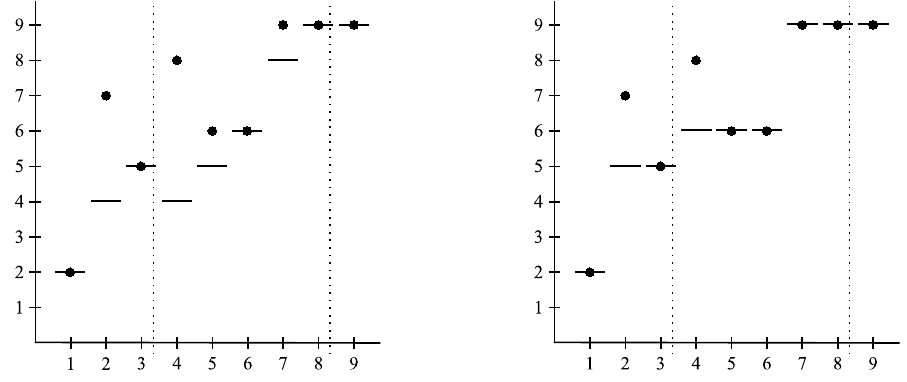}
    \caption*{Figure 5.1.  Applications of the maps $\Delta_R$ and $\Xi_R$ to our example $\beta \in U_R(9)$.}
  \end{center}
\end{figure}

Let $\beta \in U_R(n)$ and $h \in [r+1]$.  The \emph{rightmost critical index of $\beta$ in the $h^{th}$ carrel} is $q_h$.  Set $x := q_h$.  Scan the rest of the $h^{th}$ carrel $(q_{h-1}, q_h]$ from the right.  If it exists, the \emph{next critical index to the left} is $x'$, where $q_{h-1} < x' < x$ is maximal such that $\beta_{x} - \beta_{x'} > x - x'$.  Otherwise, the next critical index to the left is $x' := q_{h-1}$.  Here $q_{h-1} = 0$ occurs when $h = 1$.  Now set $x := x'$ and iterate this right-to-left scanning procedure until $x' := q_{h-1}$.  The intervals $(x', x]$ are \emph{critical intervals}; these are subintervals of the $h^{th}$ carrel.  In the second carrel in our example, the next critical index to the left after the initial (rightmost) critical index $8$ is $6$.  For $i \in (x', x]$, define $\delta_i := \beta_x - (x-i)$ and $\xi_i := \beta_x$.  Continuing to work within the second carrel in our example, we have $\delta_i = 9 - (8-i)$ and $\xi_i = 9$ for $6 < i \leq 8$.  If $x$ is a nonzero critical index, we call $\beta_x$ a \emph{critical entry}.  Overall, in our example we write the list of  nonzero critical indices as $\{ 1, 3 \hspace{1mm} | \hspace{1mm} 6, 8 \hspace{1mm} | \hspace{1mm} 9 \}$ and the list of corresponding critical entries as $\{ 2, 5 \hspace{1mm} | \hspace{1mm} 6, 9 \hspace{1mm} | \hspace{1mm} 9 \}$.  We have defined the \emph{$R$-core} $\Delta_R(\beta) := \delta$ \emph{of $\beta$} and the \emph{$R$-platform} $\Xi_R(\beta) := \xi$ \emph{of $\beta$}.  In our example $\delta = (2, 4, 5 \hspace{1mm} | \hspace{1mm} 4, 5, 6, 8, 9 \hspace{1mm} | \hspace{1mm} 9 )$ and $\xi = (2, 5, 5 \hspace{1mm} | \hspace{1mm} 6, 6, 6, 9, 9 \hspace{1mm} | \hspace{1mm} 9 )$;  see Table 6.1 for a larger example.  (The $R$-ceiling map of \cite{PW3} is the restriction of this $R$-platform map to a subset $UG_R(n)$ of $U_R(n)$.)

An $R$-\emph{tuple bounded by its platform} is an upper $R$-tuple $\beta$ such that  $\beta \leq \Xi_R(\beta)$.  Let $UBP_R(n) \subseteq U_R(n) $ denote the set of upper $R$-tuples bounded by their platforms.  The example $\beta$ above is not bounded by its platform, since $\beta_i > \xi_i$ for $i = 2$ and $4$.  Here we summarize several aspects of, and a few immediate consequences from, the definitions above:

\begin{fact}\label{critical}Let $\beta \in U_R(n)$.  Set $\delta := \Delta_R(\beta)$ and $\xi := \Xi_R(\beta)$.  Let $x', x \in [0,n]$ be such that $(x', x]$ is a critical interval for $\beta$.

\vspace{.5pc}\noindent (i) $\delta_x = \beta_x = \xi_x$.

\vspace{.5pc}\noindent (ii) For $i \in (x', x)$ one has $\delta_{i} = \delta_{i+1} - 1$.  Hence $\delta \in UI_R(n)$.

\vspace{.5pc}\noindent (iii) For $i \in (x', x)$ one has $\xi_i = \xi_x$.  Hence $\xi \in UF_R(n)$.

\vspace{.5pc}\noindent (iv) $\delta \leq \beta$ and $\delta \leq \xi$.

\vspace{.5pc}\noindent (v) If $\beta \in UBP_R(n)$, then $\delta \leq \beta \leq \xi$.

\vspace{.5pc}\noindent (vi) $\Delta_R(\delta) = \delta$ and $\Xi_R(\xi) = \xi$.  \end{fact}

\noindent Hence we can view the $R$-tuples $\delta$ and $\xi$ as being concatenations respectively of ``staircases'' and ``plateaus'' over the critical intervals for $\beta$.

It is notable when the rightmost critical entry in each carrel (which is automatically the last entry of the carrel) does not exceed the leftmost critical entry in the next carrel:  A \emph{gapless $R$-tuple} is an $R$-increasing upper tuple $\gamma$ such that for $h \in [r]$ we have $\gamma_{q_h} \leq \gamma_x$, where $x$ is the smallest critical index larger than $q_h$.  Let $UG_R(n) \subseteq UI_R(n)$ denote the set of gapless $R$-tuples.  The example $\delta$ above is a gapless $R$-tuple.  More generally, a \emph{gapless core $R$-tuple} is \emph{any} upper $R$-tuple $\eta$ such that for $h \in [r]$ we have $\eta_{q_h} \leq \eta_x$, where $x$ is the smallest critical index larger than $q_h$.  Let $UGC_R(n) \subseteq U_R(n)$ denote the set of gapless core $R$-tuples.  The example $\beta$ above is a gapless core $R$-tuple.  The following two facts explain these terminologies.  The routine verifications appeared as Parts (iii) and (ii) of Proposition 4.2 in \cite{PW3}.

\begin{fact}\label{gapless}

\vspace{.5pc}\noindent (i) Let $\gamma \in UI_R(n)$.  It is the case that $\gamma$ is a gapless $R$-tuple if and only if:  Whenever there exists $h \in [r]$ with $\gamma_{q_h} > \gamma_{q_h+1}$, then $\gamma_{q_h} - \gamma_{q_h+1} + 1 =: s \leq q_{h+1} - q_h$ and the first $s$ entries of the $(h+1)^{st}$ carrel $(q_h, q_{h+1} ]$ are $\gamma_{q_h}-s+1, \gamma_{q_h}-s+2, ... , \gamma_{q_h}$.

\vspace{.5pc}\noindent (ii) Let $\beta \in U_R(n)$.  Then $\beta \in UGC_R(n)$ if and only if $\Delta_R(\beta) \in UG_R(n)$. \end{fact}

\noindent Part (i) can be re-expressed as: whenever $\gamma_{q_h} > \gamma_{q_{h}+1}$, the leftmost staircase within the $(h+1)^{st}$ carrel must contain an entry equal to $\gamma_{q_h}$ (and hence there are no ``gaps'').

\vspace{.5pc}
\begin{figure}[h!]
\begin{center}
\setlength{\tabcolsep}{5pt}
\begin{tabular}{ccccc}
$UG$ & $\subseteq$ & $UI$ & $\subseteq$ & $UBP$\\[3pt]
 &  & $UF$ & $\subseteq$ & $UBP$ \\[3pt]
 &  & $UF$ & $\subseteq$ & $UGC$ \\[3pt]
 $UG$ & & $\subseteq$ & & $UGC$ \\[3pt]
\end{tabular}
\end{center}
\end{figure}

\vspace{-1pc}\centerline{Table 5.1.  Containments of the form $UX_R(n) \subseteq UY_R(n)$ for sets of upper $R$-tuples.}

\vspace{.5pc}

The five containments displayed in Table 5.1 follow from the definitions and the two facts; transitivity also yields the sixth containment $UG_R(n) \subseteq UBP_R(n)$.    When all column lengths are distinct, that is when $R = [n-1]$, one has $UBP_R(n) = UI_R(n) = U_R(n)$ and $UGC_R(n) = UG_R(n) = UF_R(n)$.  Hence $UGC_R(n) \cap UBP_R(n) = UF_R(n)$ here.

We relate the concepts above to the tableaux we are considering.  Given our fixed partition $\lambda$, find its set $R_\lambda$ of distinct column lengths.  Rewrite the subscript `$R_\lambda$' as `$\lambda$'.  Fix $\beta \in U_\lambda(n)$.  As was explained in Section 8 of \cite{PW3}, under value-wise comparison the row bound set $\mathcal{S}_\lambda(\beta)$ of tableaux has a unique maximal element $Q_\lambda(\beta) =: Q$.  To convert the tableau in Figure 3.1 to the maximal tableau $Q$ for that example first increase the row end value `5' to `6' and then increase nearly all of the values in each row to the row end value for that row.  However, the first values in the fourth and fifth rows are `5' and `6' and the first four values in the first row are `2'.  Visualizing how $Q$ is determined from $\beta$ motivated our definition of the critical entries of $\beta$:  The bottom value of $Q$ in a cliff of the shape $\lambda$ is the carrel-ending critical entry of $\beta$ in the corresponding carrel.  Moving up within that cliff, the semistandard condition will decrement the ending values in rows of $Q$ by 1 at each higher row until there is a precipitous drop in the given bounds in $\beta$.  At such a juncture the next critical entry to the left in $\beta$ is present in $Q$ as a row ending value.  In this manner it is seen that $\Delta_\lambda(\beta)$ is the $R_\lambda$-tuple of these row ending values of $Q$.

\section{Necessary condition for nonpermutability}\label{necc}

The two lemmas obtained here give necessary conditions for  an upper $R$-tuple $\beta$ to yield a terminal pair $(\lambda, \beta)$ that is nonpermutable.

For a determinant example pertinent to the first lemma, take $n := 3, \lambda := (1,1,0)$, and $\beta := (3,2,3)$.  Here $\Delta_\lambda(\beta) = (1,2 \hspace{1mm} | \hspace{1mm} 3)$ and $\Xi_\lambda(\beta) = (2, 2 \hspace{1mm} | \hspace{1mm} 3)$.  Note that $\beta \in UGC_\lambda(n) \backslash UBP_\lambda(n)$, and so this lemma will imply that $(\lambda, \beta)$ is not nonpermutable.  Here $s_\lambda(\beta;x_1,x_2,x_3) = x_1x_2$, but the G-V determinant of Proposition \ref{prop315.6} evaluates to $x_1x_2 - x_3^2$.

\begin{lem}\label{lemma434.5}Let $\beta \in U_\lambda(n)$.  If $\beta \notin UBP_\lambda(n)$, then $(\lambda, \beta)$ fails to be nonpermutable.  \end{lem}

For a determinant example pertinent to the second lemma, take $n := 3, \lambda := (2,1,0)$, and $\beta := (3,2,3)$.  Here $\beta = \Delta_\lambda(\beta) = \Xi_\lambda(\beta)$ since $\lambda$ is strict. Note that $\beta \in UBP_\lambda(n) \backslash UGC_\lambda(n)$, and so this lemma will imply that $(\lambda, \beta)$ is not nonpermutable.  Here $s_\lambda(\beta;x,y,z) = x_1^2x_2 + x_1x_2^2 + x_1x_2x_3$, but the G-V determinant of Proposition \ref{prop315.6} evaluates to $x_1^2x_2 + x_1x_2^2 + x_1x_2x_3 - x_3^3$.

\begin{lem}\label{lemma434.6}Let $\beta \in U_\lambda(n)$.  If $\beta \notin UGC_\lambda(n)$, then $(\lambda, \beta)$ fails to be nonpermutable.  \end{lem}

Combining the contrapositives of these two lemmas gives:

\begin{prop}\label{prop434.7}Let $\beta \in U_\lambda(n)$.  If $(\lambda, \beta)$ is nonpermutable, then $\beta \in UGC_\lambda(n) \cap UBP_\lambda(n)$.  \end{prop}

Fix $\beta \in U_\lambda(n)$.  We prepare for the proofs of both lemmas by constructing some particular $n$-paths $\Lambda$ for the given $\beta$.  To see that each of these $\Lambda$ is in $\mathcal{LD}_\lambda(\beta)$, we first describe its corresponding (clearly semistandard) tableau $T$.  Launching a running example, take $n=16$ and $\lambda = (7^3\hspace{1mm} |\hspace{1mm}  5^8 \hspace{1mm} | \hspace{1mm}  3^2 \hspace{1mm} | \hspace{1mm}  1^2 \hspace{1mm} | \hspace{1mm}  0^1)$.  Here $R_\lambda = \{ 3, 11, 13, 15 \}$ and so $r = 4$ and $q_r = 15$.  Let $\beta$ be as displayed in Table 6.1.    Set $\delta := \Delta_\lambda(\beta)$.  See Table 6.1 for the $\delta$ in the example.  Consult Figure 6.1 for the tableau $T$ being constructed in the example case below.

\begin{figure}[h!]\begin{center}
\begin{tabular}{c|ccc|cccccccc|cc|cc|c}
$i$ & 1 & 2 & 3 & 4 & 5 & 6 & 7 & 8 & 9 & 10 & 11 &  12 & 13 & 14 & 15 & 16 \\ \hline
$\beta_i$ & 5 & 5 & 8 & 5 & 12 & 13 & 9 & 11 & 11 & 15 & 15 &  16 & 16 & 14 & 16 & 16 \\
$\delta_i$ & 4 & 5 & 8 & 5 & 7 & 8 & 9 & 10 & 11 & 14 & 15 & 15 & 16 & 14 & 16 & 16 \\
$\xi_i$ & 5 & 5 & 8 & 5 & 11 & 11 & 11 & 11 & 11 & 15 & 15 & 16 & 16 & 14 & 16 & 16 \\
\end{tabular}
\caption*{Table 6.1.  An example $\beta$, with $\delta := \Delta_\lambda(\beta)$ and $\xi := \Xi_\lambda(\beta)$}
\end{center}
\end{figure}

We construct one $n$-path $\Lambda$ for each $d \in [q_r]$.  As $d$ varies, these $\Lambda$ will differ by the location of a transition from the most elementary paths for ``early'' values of $i$ to some more carefully crafted paths $\Lambda_i$ for ``middle'' values of $i$.  In our example take $ d = 9 \in [15]$.  For $i \in (0,d-1]$ set $T_j(i) := i$ for $j \in [\lambda_i]$.  These values are as small as possible.  The first $d-1$ component paths $\Lambda_i$ of $\Lambda$ corresponding to these top $d-1$ rows of $T$ are described with the top entry in Table 6.2.  Figure 6.2 uses dotted lines to display $\Lambda_3, \Lambda_4, ... , \Lambda_{16}$ for our example $\beta$; of these $\Lambda_3, \Lambda_4, ... , \Lambda_8$ are early ``elementary'' paths.  On the ending longitudes of the paths, the big (green) dots in this figure denote the depths $\delta_i$ that are taken from the

\renewcommand{\arraystretch}{0.8}
\begin{figure}[h!]
\begin{center}
\begin{tabular}{ccccccc}
1 & 1 & 1 & 1 & 1 & 1 & 1\\[3pt]
2 & 2 & 2 & 2 & 2 & 2 & 2\\[3pt]
3 & 3 & 3 & 3 & 3 & 3 & 3\\[3pt]
4 & 4 & 4 & 4 & 4\\[3pt]
5 & 5 & 5 & 5 & 5\\[3pt]
6 & 6 & 6 & 6 & 6\\[3pt]
7 & 7 & 7 & 7 & 7\\[3pt]
8 & 8 & 8 & 8 & 8\\[3pt]
9 & 9 & 9 & 11 & 11\\[3pt]
10 & 10 & 10 & 14 & 14\\[3pt]
11 & 11 & 11 & 15 & 15 \\[3pt]
12 & 15 & 15 \\[3pt]
13 & 16 & 16 \\[3pt]
14\\[3pt]
16\\[3pt]
\end{tabular}
\caption*{Figure 6.1.  Example tableau of shape $(7,7,7 \hspace{1mm} |\hspace{1mm} 5,5,5,5,5,5,5,5 \hspace{1mm} |\hspace{1mm} 3,3 \hspace{1mm} | \hspace{1mm} 1, 1\hspace{1mm} |\hspace{1mm} 0)$}
\end{center}
\end{figure}

\noindent   $\lambda$-core $\delta$.  The more nuanced paths in the middle region are indexed by $i \in (d-1, q_r]$.  Let $h \in [r]$ be such that $i \in (q_{h-1}, q_h]$.  For $j \in  [\lambda_{q_{h+1}}]$, above the values in the next (shorter) rows of $T$, set $T_j(i):=i$.  These values are still as small as possible.  For $j \in (\lambda_{q_{h+1}}, \lambda_{q_h}]$ set $T_j(i) := \delta_i$.  These ``overhanging'' values are as large as possible for the given $\beta$.  The nuanced paths $\Lambda_i$ corresponding to these middle rows of $T$ are described with the middle entry in Table 6.2.  Continuing the example, the lowest seven non-null dotted line paths $\Lambda_9, ... , \Lambda_{15}$ are of this nuanced type.  For all values of $d$ we ``fill out'' to an $n$-path $\Lambda$ in the last carrel of $\beta$:  For the ``late'' values $i \in (q_r,n]$ set $T_j(i) := \delta_i$ $(=i)$ for $j \in [\lambda_i]$.  These values are as large as possible for the given $\beta$;  the paths $\Lambda_i$ corresponding to these bottom rows are described with the bottom entry in Table 6.2.  In the example, only the last (null) path is of this third type.

\renewcommand{\arraystretch}{1}

\begin{figure}[h!]
\begin{center}
\begin{tabular}{c|c}
$i \in$ & $\Lambda_i$ \\[2pt]
\hline $(0,d-1]$ & $(n-i,i) \rightarrow (\lambda_i + n-i,i) \downarrow (\lambda_i+n-i, \delta_i) \downarrow (\lambda_i+n-i, \beta_i)$ \\[4.5pt]
$(d-1, q_r]$ &  $(n-i,i) \rightarrow (\lambda_{q_{h+1}}+n-i,i) \downarrow (\lambda_{q_{h+1}}+n-i, \delta_i) \rightarrow (\lambda_i + n - i, \delta_i) \downarrow (\lambda_i+n-i, \beta_i)$ \\[4.5pt]
$(q_r, n]$ & $(n-i,i) \rightarrow (\lambda_i+n-i, \delta_i) \downarrow (\lambda_i+n-i, \beta_i)$ \\[4.5pt]
\end{tabular}
\caption*{Table 6.2.  Original paths for the proofs of Lemmas 6.1 and 6.2}
\end{center}
\end{figure}

\begin{proof}[of Lemma 6.1] Let $\beta \in U_\lambda(n) \backslash UBP_\lambda(n)$.  By rewiring some of the paths within one of the $n$-paths constructed above, we will construct a \emph{disjoint} $n$-path $\Lambda^\prime := (\Lambda_1^\prime, ... , \Lambda_n^\prime)$ whose respective sinks form a nontrivial permutation $\pi$ of the original ordered terminals. Set $\delta := \Delta_\lambda(\beta)$ and $\xi := \Xi_\lambda(\beta)$.  Since $\xi_i = n$ for $i \in (q_r,n]$, the failure of boundedness for $\beta$ cannot occur in this last carrel of the $\lambda$-tuple $\beta$.  See Table 6.1 for $\xi$ in the running example.  Let $h \in [r]$ be such that there exists $t \in (q_{h-1}, q_h]$ such that $\beta_t > \xi_t$, and then let $c \in (q_{h-1}, q_h]$ be maximal such that $\beta_c > \xi_c$.  So $c$ is not a critical index of $\beta$ by Fact \ref{critical}(i), since $\beta_c \neq \xi_c$.  Let $d$ be the smallest critical index in $(q_{h-1},q_h]$ such that $c < d$.  In the example we have $\beta_6 = 13 > 11 = \xi_6$ with $h = 2, c = 6$, and $d = 9$.  Here $\beta_c > \xi_c = \xi_d = \beta_d = \delta_d$, which implies $\beta_c \geq \delta_d + 1$.  Since $d \leq q_r$ we have $\lambda_d \geq 1$, which implies $\lambda_d + n - d - 1 \geq 0$.

\begin{figure}[h!]
  \begin{center}
    \includegraphics[scale=1]{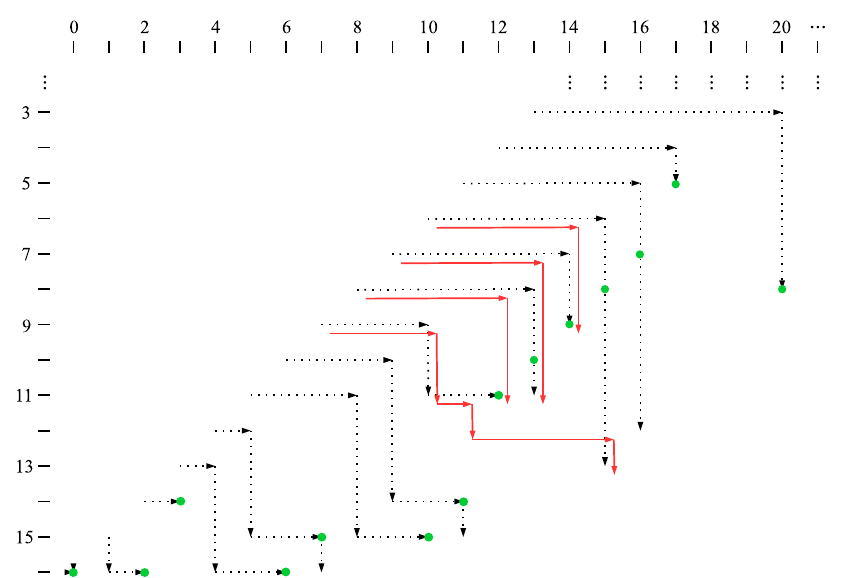}
    \caption*{Figure 6.2.  Rewiring four component paths produces a nonpermutability violation.}
  \end{center}
\end{figure}

Now refer to the $n$-path $\Lambda$ constructed above for this $d \in [q_r]$.  We will take advantage of the ``excessively'' deep sink for $\Lambda_c$.  We begin to construct $\Lambda^\prime$ by modifying the last part of $\Lambda_d$ to produce a ``rewired'' path $\Lambda_d^\prime$.  This new path $\Lambda_d'$ will sink at the sink for the old path $\Lambda_c$ after ``swooping'' beneath the terminals that are currently the sinks for $\Lambda_{c+1}, \Lambda_{c+2}, ... , \Lambda_d$.  Those terminals will now be used respectively as the sinks for the rewired $\Lambda_c', \Lambda_{c+1}', ... , \Lambda_{d-1}'$.  Look at the southernmost solid (red) path in Figure 6.2.  Rather than finishing with $... \downarrow (\lambda_{q_{h+1}}+n-d, \delta_d) \rightarrow (\lambda_d + n-d, \delta_d) = (\lambda_d+n-d, \beta_d)$, the rewired $\Lambda_d^\prime$ finishes with $...$ $\downarrow (\lambda_{q_{h+1}}+n-d, \delta_d) \rightarrow (\lambda_d + n-d-1, \delta_d) \downarrow (\lambda_d+n-d-1,$ $\delta_d+1) \rightarrow (\lambda_d + n-c, \delta_d+1) \downarrow (\lambda_c+n-c, \beta_c)$.  Here $\Lambda_d^\prime$ stops one unit short of the sink of $\Lambda_d$ which is at  $(\lambda_d+n-d, \delta_d)$, then goes one unit to the south, then turns left onto the latitude $\delta_d + 1$ and goes $d-c+1$ units to the east, and then turns right to go straight south until it sinks at $(\lambda_c+n-c, \beta_c)$.  This was the sink of $\Lambda_c$.  This new southerly edge $(\lambda_d+n-d-1, \delta_d) \downarrow (\lambda_d+n-d-1, \delta_d+1)$ in $\Lambda_d'$ is not in use by $\Lambda_{d+1}$ (or a later path):  If $d = q_h$, then the longitude at $(\lambda_d+n-d)-1$ is not used by any component of $\Lambda$ since $\lambda_d > \lambda_{d+1}$ here implies that this longitude is strictly to the east of the longitude $\lambda_{d+1}+n-d-1$ on which $\Lambda_{d+1}$ sinks.  If $d < q_h$, note that $\delta_d+1 < \delta_{d+1}$ because $d$ is a critical index.  So here the southernmost point reached by $\Lambda_d^\prime$ on its new briefly used longitude at $\lambda_d+n-d-1$ is strictly to the north of the northernmost point on this longitude used by $\Lambda_{d+1}$, which descends to the depth $\delta_{d+1}$ on the longitude $\lambda_{q_{h+1}}+n-d-1$ to the west.  Therefore $\Lambda_d'$ does not intersect $\Lambda_{d+1}$.

For either case for $d$, for $m = d-1, d-2, ... , c$ we next successively modify the finishes of $\Lambda_{d-1}, \Lambda_{d-2},$ $ ... , \Lambda_c$ to respectively produce finishes for the rewired $\Lambda_{d-1}^\prime, \Lambda_{d-2}^\prime, ... , \Lambda_c^\prime$.  Look at the other three solid (red) paths in Figure 6.2.  Let $m \in [c,d)$.  Rather than travelling the elementary path $(n-m, m) \rightarrow (\lambda_m+n-m,m) \downarrow (\lambda_m+n-m, \delta_m) \downarrow (\lambda_m+n-m, \beta_m)$, the rewired $\Lambda_m^\prime$ travels $(n-m, m) \rightarrow (\lambda_m+n-m-1,m) \downarrow (\lambda_m+n-m-1, \delta_{m+1}) \downarrow (\lambda_m+n-m-1, \beta_{m+1})$.  Here $\Lambda_m^\prime$ is finishing by turning right one step early, using one or more new southerly step(s) to reach $(\lambda_{m+1} - n - m - 1, \delta_{m+1})$, and then adopting the final (possibly empty) ``stilt'' that $\Lambda_{m+1}$ had been using to finish.  Note that $\Lambda_d'$ reaches $(\lambda_d + n-d-1, \delta_d+1)$.  Since $\lambda_{d-1} = \lambda_d$, the new southerly steps used by $\Lambda_{d-1}'$ are on the longitude $\lambda_d + n - d$, on which $\Lambda_{d-1}'$ sinks at depth $\beta_d$.  Since $\beta_d = \delta_d < \delta_d +1$, these new steps cannot intersect $\Lambda_d'$.  No intersections among these $d-c$ rewired paths occur since the right turns that are each being executed one easterly step early are being coordinated along a staircase wherein $\lambda_m = \lambda_{q_h}$.  Given the choices of $c$ and $d$, for $i \in (c,d]$ we have $\beta_i \leq \xi_i = \xi_d = \delta_d$.  So $\beta_i < \delta_d +1$ for $i \in (c,d]$.  Hence $\Lambda_m^\prime$ does not intersect $\Lambda_d^\prime$ for $m \in [c, d)$.  When $m \notin [c,d]$ set $\Lambda_m^\prime := \Lambda_m$.

We finish by ruling out intersections between rewired paths and unmodified paths.  The rewiring of $\Lambda_c$ to produce $\Lambda_c'$ deformed that path toward the southwest.  Thus $\Lambda_c'$ cannot intersect $\Lambda_{c-1}$, which did not intersect $\Lambda_c$.  The argument above that none of the rewired $\Lambda_m'$ for $m \in [c,d)$ can intersect $\Lambda_d'$ also implies that the entire rewired $\Lambda_p'$ for $p \in [c,d]$ lie in the convex hull of the steps in the paths $\Lambda_c$ and $\Lambda_d'$.  This convex region lies weakly to the west of the longitude $\lambda_c+n-c$, and the path $\Lambda_c$ forms the northeastern boundary of it.  Therefore $\Lambda_{c-1}$ cannot intersect any of the rewired paths.  For $m \in [1,c-1)$ the path $\Lambda_m$ is strictly northeast of $\Lambda_{c-1}$ through its ending longitude of $\lambda_{c-1} + n - c + 1 > \lambda_c + n - c$, the ending longitude of $\Lambda_c$.  Hence none of these $\Lambda_m$ can attain such an intersection.  We showed above that $\Lambda_d'$ and $\Lambda_{d+1}$ do not intersect.  The path $\Lambda_d'$ forms the southwestern boundary of the convex region.  So none of the rewired paths can intersect $\Lambda_{d+1}$.  All of the unmodified $\Lambda_m$ for $m \in (d+1,n]$ sink on longitudes that are to the west of the sink longitude of $\Lambda_{d+1}$.  So none of the rewired paths can intersect any such paths $\Lambda_m$.  We have permuted the original sinks by rerouting $\Lambda_d$ so that $\Lambda_d'$ ends at the sink of $\Lambda_c$ and then ``shifting'' other paths so that for $m \in [c,d)$ the rewired $\Lambda_m'$ sinks at the sink of $\Lambda_{m+1}$.  For this permutation $\pi$ of the sinks we have shown $\mathcal{LD}_\lambda(\beta;\pi) \neq \emptyset$.  \end{proof}

\begin{proof}[of Lemma 6.2]Let $\beta \in U_\lambda(n) \backslash UGC_\lambda(n)$.  Again we produce a violating $n$-path $\Lambda^\prime$ by modifying one of the $n$-paths defined at the beginning of this section.  If $\beta \notin UBP_\lambda(n)$ apply Lemma \ref{lemma434.5};  otherwise $\beta \in UBP_\lambda(n)$.  Set $\delta := \Delta_\lambda(\beta)$; this $\lambda$-tuple is in $UI_\lambda(n)$.  By Fact \ref{gapless}(ii) we have $\delta \notin UG_\lambda(n)$.  The only critical entry in the last carrel $(q_r, n]$ is $n$.  So there cannot be a failure of $\lambda$-gapless based upon having $\delta_{q_r} > n$.  Let $h \in (1, r]$ and $d \in (q_{h-1}, q_h]$ be such that $\delta$ fails to be $\lambda$-gapless based upon having $\delta_{q_{h-1}} > \delta_d$, where $d$ is the leftmost critical index in the $h^{th}$ carrel $(q_{h-1},q_h]$.  Set $c := q_{h-1}$.  Since $\beta_c = \delta_c$ at the critical index $c$, we have $\beta_c \geq \delta_d +1$.  Here $\beta_c$ is again ``excessively'' deep, as in the preceding proof.  As in that proof, in each of two cases we will rewire some of the component paths in one of the $n$-paths constructed at the beginning of this section.  Since $d \leq q_r$, in each case we have $\lambda_d \geq 1$.  This implies $\lambda_d+n-d-1 \geq 0$.  Again the two cases will be $d = q_h$ and $d < q_h$.  In each of these two cases we will refer to the $n$-path $\Lambda$ constructed at the beginning of this section for the value of $d$ at hand.  The facts obtained above allow us to now rewire the path $\Lambda_d$ in each case to produce a path $\Lambda_d^\prime$ in essentially the same fashion as in the previous proof.  The only difference is that the rewired $\Lambda_d^\prime$ now has to make $\lambda_{q_{h-1}} - \lambda_{q_h}$ additional easterly steps just before it reaches its sink longitude of $\lambda_c+n-c = \lambda_{q_{h-1}}+n- q_{h-1}$.  With this in mind, construct $\Lambda_d'$ for each case as in the preceding proof.

First suppose $d = q_h$.  We can reuse the reasoning used in the `$d = q_h$' case to see that the southerly edge on the longitude $(\lambda_d+n-d)-1$ from depth $\delta_d$ to depth $\delta_d +1$ is not in use by $\Lambda_{d+1}$ here.  Otherwise $d < q_h$.  The reasoning used in the `$d < q_h$' case before to see that the early ``jog'' to the right by $\Lambda_d'$ is acceptable can be re-used here.  So $\Lambda_d'$ does not intersect $\Lambda_{d+1}$.  For either case for $d$, the index $d$ is the smallest critical index greater than $c$.

For either case for $d$, for $m = d-1, d-2, ... , c+1$, next successively rewire $\Lambda_{d-1}, \Lambda_{d-2}, ... , \Lambda_{c+1}$ to respectively produce paths $\Lambda_{d-1}^\prime, \Lambda_{d-2}^\prime, ... , \Lambda_{c+1}^\prime$ as in the preceding proof.  Then rewire the path $\Lambda_c$ to produce a path $\Lambda_c^\prime$ in nearly the same fashion as before.  The only difference is that the rewired $\Lambda_c^\prime$ now makes $\lambda_{q_{h-1}}-\lambda_{q_h}$ fewer easterly steps just before reaching its finishing longitude of $\lambda_{c+1} + n -c -1$.  The observation in the preceding proof concerning the coordination of the right turns among the shifted $d-c$ rewired paths will need a small modification to account for this.

Now set $\xi := \Xi_\lambda(\beta)$.  For each of the two cases for $d$ the fact that $d$ is the smallest critical index larger than $c$ implies $\xi_i = \xi_d = \beta_d = \delta_d$ for $i \in (c,d]$.  Since $\beta \in UBP_\lambda(n)$, we have $\beta_i \leq \xi_i = \xi_d = \delta_d < \delta_d+1$ for $i \in (c,d]$.  Hence $\Lambda_i'$ does not intersect $\Lambda_d'$ for $i \in [c,d)$.  The rest of this proof is the same as before.  \end{proof}

\section{Sufficient condition for nonpermutability}\label{suff}

To prove the converse of Proposition \ref{prop434.7} we will need the following lemma.  Here we rename and re-index the components of $\Lambda$ according to which fixed terminals they use as sinks.

\begin{lem}\label{lemma581.5}Let $\beta \in UBP_\lambda(n)$.  Let $\pi$ be a permutation of $[n]$ and let $\Lambda \in \mathcal{LD}_\lambda(\beta;\pi)$.  Set $\delta := \Delta_\lambda(\beta)$.  For each $i \in [n]$, the component $M_i$ of $\Lambda$ that sinks at the $i^{th}$ terminal $(\lambda_i+n-i,\beta_i)$ must reach $(\lambda_i + n - i, \delta_i)$.  So it must end with the ``stilt'' $(\lambda_i+n-i, \delta_i) \downarrow (\lambda_i+n-i, \beta_i)$.  \end{lem}

\begin{proof}Let $x$ be a positive critical index for $\beta$.  Let $x'$ be the largest critical index that is less than $x$.  Since critical intervals are contained in carrels, we have $\lambda_i = \lambda_x$ for $i \in (x', x]$.  Hence the sinks for these $M_i$ lie on consecutive longitudes.  For $i \in (x', x)$ we also have $\delta_{i} = \delta_{i+1}- 1$.  Hence the points $(\lambda_x+n-i, \delta_{i})$ for $i \in (x',x]$ form a staircase, since they also lie on consecutive latitudes.  The ``reaching'' claim is true for $M_x$ since $\delta_x = \beta_x$.  Let $i$ successively decrement from $x$ to $x' + 1$ and assume the claim is true for $i' \in (i,x]$.  The source of $M_i$ is weakly to the northwest of its sink $(\lambda_i+n-i, \beta_i)$.  Set $\xi := \Xi_\lambda(\beta)$.  From Facts \ref{critical}(iii)(i) recall that $\xi_i = \xi_x = \delta_x = \beta_x$. From Fact \ref{critical}(v) recall that $\delta_i \leq \beta_i \leq \xi_i$ for $\beta \in UBP_\lambda(n)$.  Then $\delta_i \leq \beta_i \leq \delta_x$ and the blockade formed by the ``in-use'' staircase points $(\lambda_x + n - i', \delta_{i'})$ for $i' \in (i,x]$ force the path $M_i$ to reach $(\lambda_x+n-i, \delta_i)$.  Then it must finish with $(\lambda_x + n - i, \delta_i) \downarrow (\lambda_x + n - i, \beta_i)$.  \end{proof}

Stanley remarked in Theorem 2.7.1 of \cite{St1} that $(\lambda, \beta)$ is nonpermutable when $\beta$ is a flag.  Since $UF_\lambda(n) \subseteq UGC_\lambda(n) \cap UBP_\lambda(n)$, the following proposition extends that remark.

\begin{prop}\label{prop581.6}Let $\beta \in U_\lambda(n)$.  If $\beta \in UGC_\lambda(n) \cap UBP_\lambda(n)$, then $(\lambda, \beta)$ is nonpermutable.  \end{prop}

\begin{proof}Let $\beta \in UGC_\lambda(n) \cap UBP_\lambda(n)$.  Let $\pi$ be a permutation of $[n]$ that is not the identity.  There must exist at least one descent in $\pi^{-1}$.  So there exist $1 \leq i < k \leq n$ such that $\pi_i = \pi_k+1$.  Set $m := \pi_k$.  Set $\delta := \Delta_\lambda(\beta)$.  Let $\Lambda$ be an $n$-path from the standard sources $(n-m,m)$ for $m \in [n]$ to the terminals respectively listed in $\pi.(\lambda,\beta)$.  For the sake of contradicting ``nonpermutable'' suppose $\Lambda \in \mathcal{LD}_\lambda(\beta;\pi)$.  By the lemma, without loss of generality we may simplify $\Lambda$ by replacing the sequence $\beta$ of depths of its terminals with the sequence of weakly shallower depths $\delta$.  This shortens its original paths by deleting their final stilts (which cannot intersect).

As an index $p$ runs from 1 to $n$, the longitudes $\lambda_p + n - p$ of the $n$ fixed terminals (which will serve as sinks for the permuted paths) move strictly from east to west.  Visualize the paths in $\Lambda$ as being successively launched in time as their sources are scanned from northeast to southwest.  We consider the components $\Lambda_i$ and $\Lambda_k$ of $\Lambda$.  The earlier $\Lambda_i$ launches at $(n-i,i)$ and sinks at $(\lambda_{m+1}+n-m-1, \delta_{m+1})$.  The later $\Lambda_k$ launches at $(n-k,k)$ to the southwest and sinks at $(\lambda_m+n-m, \delta_m)$.  Comparing the starting and finishing longitudes for $\Lambda_k$ to those for $\Lambda_i$, we have $n-k < n-i$ and $\lambda_m + n - m > \lambda_{m+1} +n-m-1$.  So every longitude that is visited by the earlier $\Lambda_i$ is later visited by the longer $\Lambda_k$.  The earlier $\Lambda_i$ finishes at the depth $\delta_{m+1}$ on the longitude at $\lambda_{m+1}+n-m-1 =: v$.  Let's say that the later path $\Lambda_k$ first reaches the longitude at $v$ on the latitude at $z$, for some $z \geq 1$.

In each of the following two cases we prove that a (contradicting) intersection exists.
\\ (i)  First suppose that $z \leq \delta_{m+1}$.  The later $\Lambda_k$ reaches the longitude at $n-i$ weakly to the south of the earlier $\Lambda_i$ and it reaches the longitude at $\lambda_{m+1}+n-m-1$ weakly to the north of $\Lambda_i$.  So $\Lambda_k$ must intersect the continuous $\Lambda_i$.

\begin{figure}[h!]
  \begin{center}
    \includegraphics[scale=1]{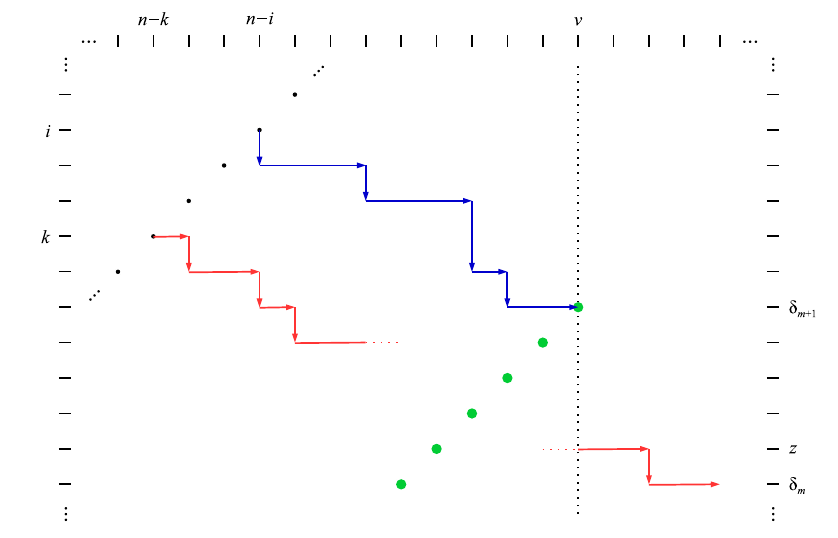}
    \caption*{Figure 7.1.  The earlier $\Lambda_i$ and the later $\Lambda_k$ respectively sink at the terminals $(v = \lambda_{m+1}+n-m-1, \delta_{m+1})$ and $(\lambda_m+n-m, \delta_m)$.}
  \end{center}
\end{figure}

\noindent (ii)  Otherwise we have $z > \delta_{m+1}$.  Since $z$ cannot exceed the finishing depth $\delta_m$ for $\Lambda_k$, we have $z \leq \delta_m$.  Hence $\delta_m > \delta_{m+1}$.  See Figure 7.1.  By Fact \ref{gapless}(ii) we know that $\delta$ is $\lambda$-gapless; in particular it is $\lambda$-increasing.  This forces $m = q_h$ for some $h \in [r]$.  Set $s := \delta_m - \delta_{m+1}+1$.  Since $\delta$ is $\lambda$-gapless, by Fact \ref{gapless}(i) we have $s \leq q_{h+1} - q_h$ and $\delta_{m+1} = \delta_m - s +1$, $\delta_{m+2} = \delta_m - s + 2, ... , \delta_{m+s} = \delta_m$.  For $p \in [s]$ we have $\lambda_{q_h+p} = \lambda_{q_{h+1}}$.  Starting at the sink $(v, \delta_{m+1})$ of $\Lambda_i$ and moving to the southwest with stairsteps, we note that the $s$ points $(v, \delta_{m+1}), (v-1, \delta_{m+1}+1), ... , (v-s+1, \delta_m)$ forming a staircase are terminals that are serving as sinks for some paths other than $\Lambda_k$.  These terminals are shown by the big dots in Figure 7.1.  The source of the later $\Lambda_k$ is strictly to the southwest of the source of the earlier $\Lambda_i$ by $s$ staircase steps.  Since $\Lambda_k$ does not intersect $\Lambda_i$, it must remain strictly to the southwest of $\Lambda_i$ as it approaches $(v,z)$.  The earlier $\Lambda_i$ sinks at $(v, \delta_{m+1})$, which is the northeasternmost point on the staircase of $s$ terminals.  So the source of the later $\Lambda_k$ must be weakly to the northwest of the staircase of terminals.  This path $\Lambda_k$ must remain weakly to the north of the latitude at $z$ as it approaches $(v,z)$.  Since $\Lambda_k$ is confined to this integral-convex region with three boundaries, to sink at $(\lambda_m + n - m, \delta_m)$ it must reach a point on the staircase of terminals.  Hence $\Lambda_k$ must intersect some component path of $\Lambda$.  \end{proof}

\section{Accidental equalities; gapless core Schur polynomials are nice}\label{nice}

Unless otherwise restricted, below $\beta$ and $\beta'$ refer to arbitrary upper $R_\lambda$-tuples.  Define $gv_\lambda(\beta;x)$ to be the G-V determinant $| h_{\lambda_j-j+i}(i,\beta_j;x) |$ introduced in Proposition \ref{prop315.6}.

To motivate our supplemental remarks and results, we begin with a sequence of review and preview statements.  The necessary direction of Theorem \ref{theorem777.1} said that $\beta$ must be in $UGC_\lambda(n) \cap UBP_\lambda(n)$ to qualify for the application of the G-V method for deducing $s_\lambda(\beta;x) = gv_\lambda(\beta;x)$ in Corollary \ref{cor777.2}.  But it is conceivable that one could ``accidentally'' have $s_\lambda(\beta;x) = gv_\lambda(\beta;x)$ when $\beta \notin UGC_\lambda(n)$ even though such a $\beta$ would produce $n$-path terminal pairs for $\lambda$ that are not nonpermutable.  Next, Corollary \ref{cor777.2.5} noted that when $\beta \in UGC_\lambda(n) \backslash UBP_\lambda(n)$ one could nonetheless use the G-V method to compute $s_\lambda(\beta;x)$ by first ``pre-processing'' $\beta$ to produce an equivalent $\beta' \in UGC_\lambda(n) \cap UBP_\lambda(n)$.  Then one would have $s_\lambda(\beta;x) = gv_\lambda(\beta';x)$.  Lastly, when $R_\lambda \subset [n-1]$, to produce a given polynomial with a row bound sum there will be some freedom in the choice of $\beta$.  This goes back to the tableaux set level, since for non-strict $\lambda$ for a given $\beta \in U_\lambda(n)$ there will exist many $\beta' \in U_\lambda(n)$ such that $\mathcal{S}_\lambda(\beta') = \mathcal{S}_\lambda(\beta)$.  And then $\mathcal{S}_\lambda(\beta) = \mathcal{S}_\lambda(\beta')$ trivially implies $s_\lambda(\beta;x) = s_\lambda(\beta';x)$.  But it is conceivable that one could accidentally have $s_\lambda(\beta';x) = s_\lambda(\beta;x)$ while $\mathcal{S}_\lambda(\beta') \neq \mathcal{S}_\lambda(\beta)$.  Here and in Section \ref{effic} we investigate what is possible and what is not possible via the choice of alternate $\beta'$.

Corollary 10.4(i) of \cite{PW3} related the equalities of the forms $\mathcal{S}_\lambda(\beta) = \mathcal{S}_\lambda(\beta')$ and $s_\lambda(\beta;x) = s_\lambda(\beta';x)$:

\begin{fact}\label{setequality}If $\eta \in UGC_\lambda(n)$, then $s_\lambda(\eta;x) = s_\lambda(\beta;x)$ for some $\beta \in U_\lambda(n)$ implies $\mathcal{S}_\lambda(\eta) = \mathcal{S}_\lambda(\beta)$ and $\beta \in UGC_\lambda(n)$.  \end{fact}

\noindent The proof of this relied upon the connection between ``gapless core Schur polynomials'' (defined below) and Demazure characters studied in that paper.  Hence if $\beta \notin UGC_\lambda(n)$, then there does not exist $\eta \in UGC_\lambda(n)$ such that accidentally $s_\lambda(\beta;x) = s_\lambda(\eta;x)$.  But for $\beta, \beta' \in U_\lambda(n) \backslash UGC_\lambda(n)$ one could at this point in time conceivably have $s_\lambda(\beta;x) = s_\lambda(\beta';x)$ when $\mathcal{S}_\lambda(\beta) \neq \mathcal{S}_\lambda(\beta')$, as noted in Problem 10.5 of \cite{PW3}.

Next we relate equalities of the form $s_\lambda(\beta;x) = s_\lambda(\beta';x)$ and $s_\lambda(\beta;x) = gv_\lambda(\beta';x)$.  Obviously $s_\lambda(\beta;x) = s_\lambda(\beta';x)$ and $s_\lambda(\beta';x) = gv_\lambda(\beta';x)$ imply $s_\lambda(\beta;x) = gv_\lambda(\beta';x)$.  This reasoning is used in the next section to confirm Corollary \ref{cor777.2.5} when $\beta \in UGC_\lambda(n)$.  More generally, when can we obtain $s_\lambda(\beta;x) = gv_\lambda(\beta';x)$?  This conclusion is the most interesting when $\beta' \in UGC_\lambda(n) \cap UBP_\lambda(n)$, since $UGC_\lambda(n) \cap UBP_\lambda(n)$ contains all of the upper $\lambda$-tuples for which the G-V method is valid.  Corollary \ref{fctinterval} below generalizes Corollary \ref{cor777.2.5}.  It is impossible to have $s_\lambda(\beta;x) = gv_\lambda(\beta';x)$ for $\beta' \in UGC_\lambda(n) \cap UBP_\lambda(n)$ when $\beta \notin UGC_\lambda(n)$:  Then $gv_\lambda(\beta';x) = s_\lambda(\beta';x)$, but $s_\lambda(\beta;x) = s_\lambda(\beta';x)$ is not possible here.  So in the restricted context of $\beta' \in UGC_\lambda(n) \cap UBP_\lambda(n)$ we have a complete answer to the question above with the statements:  when $\beta \notin UGC_\lambda(n)$ having $s_\lambda(\beta;x) = gv_\lambda(\beta';x)$ is impossible, when $\beta \in UGC_\lambda(n)$ then Corollary \ref{fctinterval} below describes the $\beta'$ for which $s_\lambda(\beta;x) = gv_\lambda(\beta';x)$, and when $\beta \in UGC_\lambda(n) \cap UBP_\lambda(n)$ the stronger result $s_\lambda(\beta;x) = gv_\lambda(\beta;x)$ is Corollary \ref{cor777.2}.  For those who are willing to accept ``accidental'' (irrespective of nonpermutability) equalities, then the following questions of necessity for being equal to the G-V determinant are open:

\begin{prob}\label{open}Let $\beta \in U_\lambda(n)$.

\noindent (i)  Suppose $\beta \notin UGC_\lambda(n) \cap UBP_\lambda(n)$.  Is it possible to have $s_\lambda(\beta;x) = gv_\lambda(\beta;x)$?

\noindent (ii)  Suppose $\beta \notin UGC_\lambda(n)$.  Is it possible to have $s_\lambda(\beta;x) = gv_\lambda(\beta';x)$ for some $\beta' \in U_\lambda(n)$?  \end{prob}

We add to one of the themes of \cite{PW3}.  There we showed that the row bound sums $s_\lambda(\eta;x)$ for $\eta \in UGC_\lambda(n)$ behaved more nicely than the row bound sums $s_\lambda(\beta;x)$ for $\beta \in U_\lambda(n) \backslash UGC_\lambda(n)$.  Hence we awarded the name \emph{gapless core Schur polynomials} to the former polynomials, for which the following can be said:  When $\eta \in UGC_\lambda(n)$, the polynomial $s_\lambda(\eta;x)$ arises as a previously-studied flag Schur polynomial and as a Demazure polynomial, as shown in Proposition 8.1 and Theorem 9.1(i) of \cite{PW3}.  (In fact, its tableau set coincides with the relevant set of Demazure tableaux.)  When $\eta' \in UGC_\lambda(n)$ as well, then $s_\lambda(\eta;x)$ and $s_\lambda(\eta';x)$ cannot be equal unless $\mathcal{S}_\lambda(\eta) = \mathcal{S}_\lambda(\eta')$.  The polynomial $s_\lambda(\eta;x)$ can be computed with a G-V determinant (though this may require pre-processing).  The distinct polynomials arising as gapless core Schur polynomials $s_\lambda(\eta;x)$ are counted by the parabolic Catalan numbers, as shown in Theorem 13.1(iv) of \cite{PW3}.  In contrast, the row bound sums $s_\lambda(\beta;x)$ for $\beta \in U_\lambda(n) \backslash UGC_\lambda(n)$ seem to be of dubious value.  They cannot arise as flag Schur polynomials or as Demazure polynomials, as shown in Corollary 10.4(i) and Theorem 10.3 of \cite{PW3}.  It is not known if one can have $s_\lambda(\beta;x) = s_\lambda(\beta';x)$ with $\mathcal{S}_\lambda(\beta) \neq \mathcal{S}_\lambda(\beta')$ for $\beta' \in U_\lambda(n) \backslash UGC_\lambda(n)$, nor is it known if $s_\lambda(\beta;x)$ can somehow be computed with an $n \times n$ determinant.

\section{Equivalence and efficiency}\label{effic}

When the partition $\lambda$ is not strict, more than one upper $\lambda$-tuple $\beta$ can specify the same tableau set $\mathcal{S}_\lambda(\beta)$ or the same row bound sum $s_\lambda(\beta;x)$.  And when $\beta$ is more specifically a gapless core $\lambda$-tuple, more than one bounded gapless core $\lambda$-tuple $\beta'$ can be used to express $s_\lambda(\beta;x)$ with $gv_\lambda(\beta';x)$.  We introduce equivalence relations to describe how much freedom one has when choosing alternate upper $\lambda$-tuples $\beta'$ for the computation of $s_\lambda(\beta;x)$.  Within each class we describe the upper $\lambda$-tuples that are valid for the application of the G-V method, and then we identify the upper $\lambda$-tuple that is the most efficient for evaluating the G-V determinant.

In Section 8 of \cite{PW3} we introduced an equivalence relation on $U_\lambda(n)$ by defining $\beta \approx_\lambda \beta'$ when $\mathcal{S}_\lambda(\beta) = \mathcal{S}_\lambda(\beta')$.  There Proposition 8.2 stated that this equivalence relation was the same as the equivalence relation $\sim_R$ defined on $U_R(n)$ in that Section 5 when $R := R_\lambda$.  Lemma 5.1 of \cite{PW3} said that this earlier relation $\sim_R$ could be defined entirely in terms of upper $\lambda$-tuples using the map $\Delta_{R_\lambda} =: \Delta_\lambda$.  The following background facts can be confirmed using Proposition 5.2(ii)(iii), Proposition 4.3(ii), and Lemma 5.1(i) of \cite{PW3}, given the non-essential definitions of ``$\lambda$-canopy'' tuple and ``$\lambda$-floor'' flag in Parts (iv) and (v) of that Definition 3.1:

\begin{fact}\label{OldIntervals}  The equivalence classes for the restrictions of $\approx_\lambda$ and $\sim_R$ to $UGC_\lambda(n)$ and to $UF_\lambda(n)$ are:

\noindent(i) In $UGC_\lambda(n)$ these subsets are the intervals in $U_\lambda(n)$ of the form $[ \gamma, \kappa ]$, where $\gamma$ is a gapless $\lambda$-tuple and $\kappa$ is the unique ``$\lambda$-canopy'' tuple such that $\gamma = \Delta_\lambda(\kappa)$.

\noindent(ii) In $UF_\lambda(n)$ these subsets are the intervals in $UF_\lambda(n)$ of the form $[ \tau, \xi ]$, where $\tau$ is a ``$\lambda$-floor'' flag and $\xi := \Xi_\lambda(\tau)$.  \end{fact}

The equivalence classes in our most favored set $UGC_\lambda(n) \cap UBP_\lambda(n)$ of upper $\lambda$-tuples are described by pairing the minimum elements of Part (i) with maximum elements of the same ``$\lambda$-ceiling'' kind as in Part (ii):

\begin{prop}\label{newinterval}  The equivalence classes for the restrictions of $\approx_\lambda$ and $\sim_R$ to $UGC_\lambda(n) \cap UBP_\lambda(n)$ are the intervals in $U_\lambda(n)$ of the form $[\gamma, \xi]$, where $\gamma$ is a gapless $\lambda$-tuple and $\xi$ is the upper $\lambda$-flag $\Xi_\lambda(\gamma)$.  \end{prop}

\begin{proof} The restricted classes are the intersections of the classes in $UGC_\lambda(n)$ with $UBP_\lambda(n)$.  By Fact \ref{OldIntervals}(i), these restricted classes have the form $[\gamma, \kappa] \cap UBP_\lambda(n)$.  To consider one of these, fix $\gamma \in UG_\lambda(n)$ and let $\kappa$ be the unique $\lambda$-canopy tuple in $U_\lambda(n)$ such that $\Delta_\lambda(\kappa) = \gamma$.  Set $\xi := \Xi_\lambda(\gamma)$.  We have $\xi \in UF_\lambda(n)$ by Fact \ref{critical}(iii).  Lemma 5.1 of \cite{PW3} says that two upper $\lambda$-tuples $\beta$ and $\beta'$ are equivalent if and only if $\Delta_\lambda(\beta) = \Delta_\lambda(\beta')$.  So $\beta \in [\gamma, \kappa]$ for $\beta \in U_\lambda(n)$ if and only if $\Delta_\lambda(\beta) = \Delta_\lambda(\gamma)$.  Our simultaneous definitions for the maps $\Delta_\lambda$ and $\Xi_\lambda$ imply that $\Xi_\lambda(\beta) = \Xi_\lambda(\beta')$ for two upper $\lambda$-tuples $\beta$ and $\beta'$ if and only if $\Delta_\lambda(\beta) = \Delta_\lambda(\beta')$.  The class of $\gamma$ in $U_\lambda(n)$ is $[\gamma, \kappa]$.  Hence $\beta \in [\gamma, \kappa]$ for $\beta \in U_\lambda(n)$ if and only if $\Xi_\lambda(\beta) = \Xi_\lambda(\gamma) =: \xi$.  The definition of $UBP_\lambda(n)$ says $\beta \in UBP_\lambda(n)$ for $\beta \in [\gamma, \kappa]$ if and only if $\beta \leq \xi$.  Since $\Xi_\lambda(\xi) = \xi$ by Fact \ref{critical}(vi), we can deduce $\xi \in [\gamma, \kappa]$ and thus $\xi \leq \kappa$.  Therefore $\beta \in [\gamma, \xi]$ for $\beta \in U_\lambda(n)$ if and only if $\beta \in [\gamma, \kappa] \cap UBP_\lambda(n)$.  Clearly $[\gamma, \xi] \subseteq UGC_\lambda(n) \cap UBP_\lambda(n)$.  \end{proof}

\noindent The equivalence classes in Proposition \ref{newinterval} are induced on $UGC_\lambda(n) \cap UBP_\lambda(n)$ by any one of the following equalities:  $\mathcal{S}_\lambda(\beta) = \mathcal{S}_\lambda(\beta')$, $s_\lambda(\beta;x) = s_\lambda(\beta';x)$, $s_\lambda(\beta) = gv_\lambda(\beta')$, or $gv_\lambda(\beta) = gv_\lambda(\beta')$.  Since it was noted in Section \ref{advdefns} that $UF_\lambda(n) \subseteq UGC_\lambda(n)$ and $UF_\lambda(n) \subseteq UBP_\lambda(n)$, if one is interested only in flags there is no need to consider how the classes for $\approx_\lambda$ restrict to $UF_\lambda(n) \cap ( UGC_\lambda(n) \cap UBP_\lambda(n) ) = UF_\lambda(n)$.

\begin{cor}\label{fctinterval}Let $\eta \in UGC_\lambda(n)$ and $\eta' \in UGC_\lambda(n) \cap UBP_\lambda(n)$.  Then $s_\lambda(\eta;x) = gv_\lambda(\eta';x)$ if and only if $\eta' \in [\Delta_\lambda(\eta), \Xi_\lambda(\eta)]$.  Here $\Delta_\lambda(\eta)$ is a gapless $\lambda$-tuple and $\Xi_\lambda(\eta)$ is an upper $\lambda$-flag.  This nonempty interval is contained in $UGC_\lambda(n) \cap UBP_\lambda(n)$. \end{cor}

\begin{proof}Let $\eta \in UGC_\lambda(n)$.  Set $\gamma := \Delta_\lambda(\eta)$ and $\xi := \Xi_\lambda(\gamma) \in UF_\lambda(n)$.  Since $\eta \in UGC_\lambda(n)$, by Fact \ref{gapless}(ii) we have $\gamma \in UG_\lambda(n)$ and by Fact \ref{critical}(iv) we have $\gamma \leq \xi$.  So $[\gamma, \xi]$ is of the class form $[\gamma', \xi']$ in the proposition, and hence it is contained in $UGC_\lambda(n) \cap UBP_\lambda(n)$ as well as being nonempty.  Since $\Delta_\lambda(\gamma) = \gamma := \Delta_\lambda(\eta)$ by Fact \ref{critical}(vi), by Lemma 5.1 of \cite{PW3} we can deduce that $\eta \approx_\lambda \gamma$.  Let $\eta'$ be in the class $[\gamma, \xi]$ for $\approx_\lambda$ of $\gamma$ in $UGC_\lambda(n) \cap UBP_\lambda(n)$.  Then $\mathcal{S}_\lambda(\eta) = \mathcal{S}_\lambda(\eta')$ and so $s_\lambda(\eta;x) = s_\lambda(\eta';x)$.  Corollary \ref{cor777.2} says that $s_\lambda(\eta';x) = gv_\lambda(\eta';x)$.  The special case of $s_\lambda(\eta;x) = gv_\lambda(\eta';x)$ for $\eta' := \gamma$ was stated as Corollary \ref{cor777.2.5}.  Conversely, let $\eta' \in UGC_\lambda(n) \cap UBP_\lambda(n)$ be such that $s_\lambda(\eta;x) = gv_\lambda(\eta';x)$.  Then $s_\lambda(\eta';x) = gv_\lambda(\eta';x)$ implies $s_\lambda(\eta;x) = s_\lambda(\eta';x)$, and so Fact \ref{setequality} implies $\mathcal{S}_\lambda(\eta) = \mathcal{S}_\lambda(\eta')$.  Hence $\eta \approx_\lambda \eta'$.  Since $\eta \approx_\lambda \gamma$ and $\eta' \in UGC_\lambda(n) \cap UBP_\lambda(n)$, we see that $\eta'$ must be in the class $[\gamma, \xi]$ for $\approx_\lambda$ of $\gamma \in UGC_\lambda(n) \cap UBP_\lambda(n)$.  \end{proof}

Applying the $\lambda$-platform map $\Xi_\lambda$ to a gapless core $\lambda$-tuple $\eta$ produces an upper $\lambda$-flag $\varphi$ that is equivalent to $\eta$.  So any row bound set $\mathcal{S}_\lambda(\eta)$ for a gapless core $\lambda$-tuple also arises as a row bound set $\mathcal{S}_\lambda(\varphi)$ for an upper $\lambda$-flag.  This confirms the remark made at the end of Section \ref{nice} that every gapless core Schur polynomial has already arisen as a flag Schur polynomial.  If someone was to insist that their input to a G-V determinant must be an upper $\lambda$-flag, then at least the maximum element $\Xi_\lambda(\eta)$ of the interval of Corollary \ref{fctinterval} would be available.  However, from the viewpoint of efficient determinant evaluation, the proof of our next result should indicate that that upper $\lambda$-flag would be the worst choice from that interval of choices.

Let $\eta \in UGC_\lambda(n)$.  We say $\eta' \in UGC_\lambda(n) \cap UBP_\lambda(n)$ attains \emph{maximum efficiency} if $gv_\lambda(\eta';x) := |h_{\lambda_j-j+i}(i,\eta'_j;x)|$ has fewer total monomials among its entries than does the G-V determinant $gv_\lambda(\eta'';x)$ for any other $\eta'' \in UGC_\lambda(n) \cap UBP_\lambda(n)$ that produces $s_\lambda(\eta;x)$.

\begin{prop}\label{prop826.9}  Let $\eta \in UGC_\lambda(n)$.  The gapless $\lambda$-tuple $\Delta_\lambda(\eta)$ attains maximum efficiency among the choices in the interval $[\Delta_\lambda(\eta), \Xi_\lambda(\eta)]$ allowed by Corollary \ref{fctinterval}. \end{prop}

\begin{proof}Set $\gamma := \Delta_\lambda(\eta)$.  Let $\eta' \in (\gamma, \Xi_\lambda(\eta)]$.  For $i, j \in [n]$ the $(i,j)$-entry of the G-V determinant $gv_\lambda(\eta';x)$ has $\binom{\lambda_j-j+\eta'_j}{\lambda_j-j+i}$ monomials.  Since $\gamma < \eta'$ there exists some $m \in [n]$ such that $\lambda_m - m + \eta'_m > \lambda_m - m + \gamma_m$ and $\lambda_j - j + \eta'_j \geq \lambda_j - j + \gamma_j$ for $m \neq j \in [n]$. \end{proof}

The description of lattice paths given in the proof of Lemma \ref{lemma581.5} can be used to visualize the choices in one of the equivalence classes $[\gamma, \xi]$ of Proposition \ref{newinterval}:  These choices vary only in the lengths of their path-ending stilts.  We have not been able to convert the G-V determinant $gv_\lambda(\eta;x)$ for $\eta \in [\gamma, \xi]$ to the G-V determinant $gv_\lambda(\gamma;x)$ with naive row and column operations.  If $\lambda_n > 0$, one can factor out $(x_1x_2 \cdots x_n)^{\lambda_n}$ from $gv_\lambda(\gamma;x)$ and work with $\lambda^\prime := (\lambda_1 - \lambda_n, \lambda_2 - \lambda_n, ... , 0)$.  Going further, when there are only $p < n$ nonempty rows in the shape $\lambda$, the determinant $gv_{\lambda'}(\gamma;x)$ is equal to its upper left $p \times p$ minor.

Various combinatorial counts have been expressed with determinants of binomial coefficients.  If two families of such determinants look similar and test evaluations yield the same integers, then one first tries to relate them with row and column operations.  If that does not work, the following consequence of Corollary \ref{cor777.2.5} offers an alternative:

\begin{cor}\label{EqualDeterms}Let $\lambda$ be a partition and let $\beta, \beta'$ be gapless core $\lambda$-tuples.  If $\Delta_\lambda(\beta) = \Delta_\lambda(\beta')$ then $$\left|{ \lambda_j - j + \beta_j \choose \beta_j - i} \right| = \left|{ \lambda_j - j + \beta'_j \choose \beta'_j - i} \right|$$    \end{cor}
\noindent The hypothesis concerning the cores of $\beta$ and $\beta'$ is equivalent to $\beta \approx_\lambda \beta'$.

To illustrate efficiency for determinant entries we first adapt the G-V proof of the Jacobi-Trudi identity for infinite variables in Theorem 7.16.1 of \cite{St2} to obtain the classic form of this identity in a finite number $n$ of variables for a nonskew shape, as in Equation 2.8 of \cite{Oka}.  Working with a finite number of variables will allow us to use semistandard tableaux directly, rather than refering to the equivalence with reverse semistandard tableaux as in \cite{St2}.  Avoiding that equivalence can be accomplished with the following relabeling, which will also harmonize that proof with the vertical labeling convention in this paper:  Now label the translates of the horizontal axis in \cite{St2} from the north with $1, 2, ... , n$.  (There in Figure 7-6 from the top these are relabeled 1, 2, 3, 4.)  Reading from the east, in general the sources for the $n$ component finite paths in this adaptation become $(n-1,1), (n-2,1), ... , (0,1)$.  Since this application of Theorem 2.7.1 of \cite{St1} uses the row bounds $(n, n, ... , n)$, the $n$ terminals in the proof of Theorem 7.16.1 are $(\lambda_1 + n - 1, n), (\lambda_2 + n - 2, n), ... , (\lambda_n, n)$.  Because all paths start at depth 1 and end at depth $n$, each of the determinant entries in that theorem are complete homogeneous symmetric functions.  For $1 \leq i, j \leq n$ the $(i,j)$-entry has $\binom{\lambda_j+n+i-j-1}{n-1}$ monomials.  Now we compare the derivation of our Corollary \ref{cor777.2} in the same lattice path setting:  Choose row bounds $\beta \in UGC_\lambda(n) \cap UBP_\lambda(n)$.  Recall that our sources are $(n-1,1), (n-2,2), ... , (0,n)$ and our terminals are $(\lambda_1+n-1, \beta_1), (\lambda_2+n-2, \beta_2), ... , (\lambda_n, \beta_n)$.  Set $\delta := \Delta_\lambda(\beta)$.  When $\lambda$ is not strict the general inequality $\delta \leq \beta$ often becomes strict.  For $1 \leq i, j \leq n$ the $(i,j)$-entry of our determinant in Corollary \ref{cor777.2.5} has only $\binom{\lambda_j + \delta_j - j}{\delta_j - i}$ monomials.  Proposition \ref{prop826.9} stated that these are the most efficient entries possible in this context.  These efficiencies have been obtained by first deleting the ``initial stilts'' for the $m^{th}$ component paths $\Lambda_m$ for $m \in [n]$ in the adaptation of \cite{St2} above, starting instead at our sources.  Those stilts dropped from $(n-m,1)$ to $(n-m,m)$.  For a general gapless core Schur polynomial one usually chooses $\beta < (n, n,  ... , n)$.  Changing the row bounds from $(n, n, ... , n)$ to $\beta$ deletes portions of some ``ending stilts''.  Using Lemma \ref{lemma581.5} it can be seen that one may as well terminate the paths at the terminals $(\lambda_1 + n - 1, \delta_1), (\lambda_2 + n - 2, \delta_2), ... , (\lambda_n, \delta_n)$ in the G-V derivation of Corollary \ref{cor777.2}.  Then also changing the row bounds from $\beta$ to $\delta$ deletes the rest of the unnecessary ``ending stilts'', further shortening the paths by $\beta_m - \delta_m$ for $m \in [n]$.  These ending stilts dropped from $(\lambda_m+n-m, \delta_m)$ to $(\lambda_m+n-m, \beta_m)$, and then from $(\lambda_m+n-m, \beta_m)$ to $(\lambda_m+n-m, n)$.  For a simple numerical comparison consider the following extreme example:  Let $p \geq 1$ and set $n := 2p+1$.  Let $q \geq 1$.  Consider the $n$-tuples $\lambda := (q, q, ... , q)$ and $\beta := (2p+1, 2p+1, ... , 2p+1)$.  So $\delta = (1, 2, 3, ... , 2p, 2p+1)$.  Here the $(i,j)$-entry of the Jacobi-Trudi determinant has $\binom{q+2p+i-j}{2p}$ monomials.  In contrast our $(i,j)$-entry has $\binom{q}{j-i}$ monomials.  When $i = j$ this comparison becomes $\binom{q+2p}{2p} > \binom{q}{0} = 1$.

In the first paper \cite{PW2} of this series we defined the parabolic Catalan number $C_n^\lambda$ to be the number of ``$\lambda$-312-avoiding permutations''.  There in Theorem 9.1(iii) we noted that this is also the number of gapless $\lambda$-tuples.  Given this, the following result is a consequence of Propositions \ref{newinterval} and \ref{prop826.9}.  It was previewed as Part (vi) of Theorem 13.1 of \cite{PW3}:

\begin{cor}\label{cor826.12}The number of valid upper $\lambda$-tuple inputs to the G-V determinant expression for flag Schur polynomials on the shape $\lambda$ that attain maximum efficiency is $C_n^\lambda$.  \end{cor}

\noindent For a sequence of examples, let $m \geq 1$ and set $n := 2m$.  Suppose $\lambda$ is a partition whose shape's set of column lengths is $R_\lambda = \{ 2, 4, ... , 2m-2 \}$.  Then the number of maximum efficiency inputs here is given by the member of Sequence A220097 of the OEIS \cite{Slo} that is indexed by $m$.

\section{Determinant expression for some Demazure characters}\label{Dem}

At the end of Section 10 of \cite{PW3} we promised to give a determinant expression for certain $GL(n)$ Demazure characters (key polynomials) here.  As noted in that section of \cite{PW3}, general Demazure characters $d_\lambda(\pi;x)$ for $GL(n)$ can be defined with divided differences or as a sum of $x^{\Theta(T)}$ over a certain set $\mathcal{D}_\lambda(\pi)$ of semistandard tableaux.  See for example \cite{PW1}.  Given that $UG_\lambda(n) \subseteq UGC_\lambda(n) \cap UBP_\lambda(n)$, the next statement follows from Theorem 10.2(ii) of \cite{PW3} and Corollary \ref{cor777.2} here.  For this result that theorem gives $d_\lambda(\pi;x) = s_\lambda(\gamma;x)$.  Consult Section 3 of \cite{PW3} for the definitions of $\lambda$-permutations and the map $\Psi_\lambda$.

\begin{cor}\label{cor777.3}Let $\lambda$ be a partition and let $\pi$ be a $\lambda$-permutation.  If $\pi$ is $\lambda$-312-avoiding, then $\Psi_\lambda(\pi) =: \gamma$ is a gapless $\lambda$-tuple and $d_\lambda(\pi;x) = | h_{\lambda_j-j+i}(i,\gamma_j;x) |$. \end{cor}

\noindent A ``less efficient'' (in the sense of our Section \ref{effic}) version of this expression appeared in the proof of Corollary 14.6 of \cite{PS} when Postnikov and Stanley applied their skew flagged Schur function determinant identity Equation 13.1 to their $\text{ch}_{\lambda,w}$.

Section 3 of \cite{KM} proposes another approach to determinant generating function enumeration of tableaux (satisfying general row bounds) that are expressed in terms of n-paths.

\vspace{1pc}\noindent \textbf{Acknowledgments.}  We thank the referees for encouraging us to improve the exposition and the title of this paper.  The second author thanks the math departments of Connecticut College and Wesleyan University for their support; portions of this paper were developed there.

\end{document}